\documentclass{article}
\usepackage{hyperref}
\usepackage{stmaryrd}
\usepackage{amsmath}
\usepackage{latexsym}
\usepackage{amssymb,amsthm}
\usepackage[margin=1in]{geometry}

\newtheorem{theorem}{Theorem}[section]
\newtheorem{corollary}{Corollary}[section]
\newtheorem{lemma}{Lemma}[section]
\newtheorem{proposition}{Proposition}[section]
\newtheorem{definition*}{Definition}[section]
\newtheorem{example*}{Example}[section]
\newtheorem{remark*}{Remark}[section]

\newenvironment{definition}{\begin{definition*}\rm}{\end{definition*}}
\newenvironment{example}{\begin{example*}\rm}{\end{example*}}
\newenvironment{remark}{\begin{remark*}\rm}{\end{remark*}}

\newcommand{\bd}{\mathbf d}

\renewcommand{\phi}{\varphi}

\font\tengoth=eufm10 at 10pt
\font\sevengoth=eufm7 at 6pt
\newfam\gothfam
\textfont\gothfam=\tengoth
\scriptfont\gothfam=\sevengoth

\newcommand{\fS}{{\mathfrak S}}

\newcommand{\g}{{\mathfrak g}}

\newcommand{\fh}{{\mathfrak h}}

\newcommand{\1}{\mathbf{1}}

\newcommand{\cA}{\mathcal{A}}
\newcommand{\cB}{\mathcal{B}}
\newcommand{\cC}{\mathcal{C}}
\newcommand{\cD}{\mathcal{D}}

\newcommand{\cH}{\mathcal{H}}
\newcommand{\cK}{\mathcal{K}}
\newcommand{\cL}{\mathcal{L}}

\newcommand{\cS}{\mathcal{S}}

\newcommand{\dd}{{\tt d}}

\newcommand{\subeq}{\subseteq}

\newcommand{\into}{\hookrightarrow}
\newcommand{\eps}{\varepsilon}

\newcommand{\N}{{\mathbb N}}
\newcommand{\Z}{{\mathbb Z}}
\newcommand{\R}{{\mathbb R}}
\newcommand{\C}{{\mathbb C}}

\newcommand{\K}{{\mathbb K}}

\newcommand{\T}{{\mathbb T}}

\renewcommand{\hat}{\widehat}

\renewcommand{\tilde}{\widetilde}

\renewcommand{\L}{\mathop{\bf L{}}\nolimits}

\newcommand{\GL}{\mathop{{\rm GL}}\nolimits}

\newcommand{\Gr}{\mathop{{\rm Gr}}\nolimits}

\newcommand{\U}{\mathop{\rm U{}}\nolimits}

\newcommand{\ad}{\mathop{{\rm ad}}\nolimits}
\newcommand{\Ad}{\mathop{{\rm Ad}}\nolimits}

\renewcommand{\Im}{\mathop{{\rm Im}}\nolimits}

\newcommand{\Hom}{\mathop{{\rm Hom}}\nolimits}

\newcommand{\Rep}{\mathop{{\rm Rep}}\nolimits}

\newcommand{\Heis}{\mathop{{\rm Heis}}\nolimits}

\newcommand{\Diff}{\mathop{{\rm Diff}}\nolimits}

\newcommand{\Vir}{\mathop{{\rm Vir}}\nolimits}

\newcommand{\End}{\mathop{{\rm End}}\nolimits}

\renewcommand{\dim}{\mathop{{\rm dim}}\nolimits}

\newcommand{\Rarrow}{\Rightarrow}
\newcommand{\nin}{\noindent} 
\newcommand{\oline}{\overline}

\newcommand{\la}{\langle}
\newcommand{\ra}{\rangle}

\newcommand{\res}{\vert}

\newcommand{\Spec}{{\rm Spec}}

\newcommand{\ssssarr}{\hbox to 15pt{\rightarrowfill}}
\newcommand{\sssarr}{\hbox to 20pt{\rightarrowfill}}
\newcommand{\ssarr}{\hbox to 30pt{\rightarrowfill}}
\newcommand{\sarr}{\hbox to 40pt{\rightarrowfill}}
\newcommand{\arr}{\hbox to 60pt{\rightarrowfill}}
\newcommand{\larr}{\hbox to 60pt{\leftarrowfill}}
\newcommand{\Arr}{\hbox to 80pt{\rightarrowfill}}

\begin{document}



\title{Smoothing operators and $C^*$-algebras 
for \\  infinite dimensional Lie groups} 

\author{Karl-Hermann Neeb, Hadi Salmasian, Christoph Zellner}

\AtEndDocument{\bigskip{\footnotesize%
  \textsc{Karl-Hermann Neeb, Department Mathematik, FAU Erlangen-N\"urnberg, Cauerstra\ss e 11, 91058 Erlangen, Deutschland;  } \par  
  \textit{E-mail address}: \texttt{neeb@math.fau.de} \par
  \addvspace{\medskipamount}
  \textsc{Hadi Salmasian, Department of Mathematics and Statistics, University of Ottawa, 585 King Edward Ave, Ottawa, ON K1N 6N5, Canada} \par
  \textit{E-mail address}: \texttt{hsalmasi@uottawa.ca} \par
  \addvspace{\medskipamount}
  \textsc{Christoph Zellner, Department Mathematik, FAU Erlangen-N\"urnberg, Cauerstra\ss e 11, 91058 Erlangen, Deutschland} \par
  \textit{E-mail address}: \texttt{primes@gmx.de}
}}




\maketitle

\begin{abstract}
{A smoothing operator for a 
unitary representation $\pi:G\to \U(\cH)$ of a (possibly infinite dimensional) Lie group $G$ is a bounded operator $A:\cH\to \cH$ whose range is  
contained in the space 
$\cH^\infty$ of smooth vectors of $(\pi,\cH)$. 
Our first main result characterizes   smoothing operators for Fr\'echet--Lie groups 
as those for which the orbit map $\pi^A \colon G \to B(\cH),\ g \mapsto \pi(g)A$ is smooth.
For unitary representations $(\pi,\cH)$ which are semibounded, 
i.e., there exists an element $x_0 \in\g$ such that all operators 
$i\dd\pi(x)$ from the derived representation, for $x $ in a neighborhood of $x_0$,  
are uniformly bounded from above, we show that 
$\cH^\infty$ coincides with the space of 
smooth vectors for the one-parameter group $\pi_{x_0}(t) = \pi(\exp tx_0)$. 
As the main application of our results on smoothing operators, we present a new 
approach to host $C^*$-algebras for infinite dimensional Lie groups, 
i.e., $C^*$-algebras  whose representations are in one-to-one correspondence with certain 
continuous unitary representations of $G$. 
We show that smoothing operators can be 
used to obtain host algebras and that 
the class of semibounded representations can be covered completely 
by host algebras. In particular, the latter class permits direct integral decompositions.}\\


\end{abstract}



\section*{Introduction} \label{sec:0}

If $G$ is a locally compact group, then a Haar measure 
on $G$ leads to the convolution algebra $L^1(G)$, and we obtain 
a $C^*$-algebra $C^*(G)$ as the enveloping $C^*$-algebra of $L^1(G)$. 
This $C^*$-algebra has the universal property that each 
(continuous) unitary representation $(\pi, {\cal H})$ of $G$ on 
a Hilbert space ${\cal H}$ defines a unique non-degenerate 
representation of $C^*(G)$ on ${\cal H}$ and, conversely, each 
non-degenerate representation of $C^*(G)$ arises from a unique 
unitary representation of $G$. This correspondence is a central tool 
in the harmonic analysis on $G$ because the well-developed theory of 
$C^*$-algebras provides a powerful machinery to study the set of all 
irreducible representations of $G$, to endow it with a natural 
topology and to understand how to decompose representations into 
irreducible ones or factor representations. 

For infinite dimensional Lie groups (modeled on locally convex spaces)  
there is no natural 
analog of the convolution algebra $L^1(G)$, so that we cannot hope 
to find a $C^*$-algebra whose representations are in one-to-one 
correspondence with all unitary representations of $G$. However, 
in \cite{Gr05} H.~Grundling introduces the notion of a {\it host algebra}  
of a topological group $G$. This is a pair $({\cal A}, \eta)$, consisting 
of a $C^*$-algebra ${\cal A}$ and a morphism 
$\eta \colon G \to \U(M({\cal A}))$ of $G$ into the unitary group 
of its multiplier algebra $M({\cal A})$ with the following property:  
For each non-degenerate representation $\pi$ of ${\cal A}$ 
and its canonical extension $\tilde\pi$ to $M({\cal A})$, 
the unitary representation $\tilde\pi \circ \eta$ of $G$ is continuous 
and determines $\pi$ uniquely. In this sense, ${\cal A}$ is hosting a 
certain class of representations of $G$. A host algebra $\cA$ is called 
{\it full} if it is hosting all continuous unitary representations of $G$.  
Now it is natural to ask to which extent infinite dimensional 
Lie groups possess host algebras. If $G = (E,+)$ is an infinite dimensional locally convex space, 
then the set of equivalence classes of irreducible unitary representations 
identifies naturally with the dual space $E'$, and since this space carries
no natural locally compact topology, one cannot expect the existence 
of a full host algebra in general. 
Therefore one is looking for host algebras that 
accommodate certain classes of continuous unitary representations. 

For a unitary representation $(\pi, \cH)$ of a 
finite dimensional Lie group $G$ and $f \in C^\infty_c(G)$, 
the operator $\pi(f) = \int_G f(g)\pi(g)\, dg$ has the nice property that 
its range consists of smooth vectors (G\aa{}rding's Theorem) and 
the image of $C^*(G)$ in $B(\cH)$ is generated by these operators. 
We take this as a starting point of our construction of host algebras: 
We call a bounded operator 
$A \in B(\cH)$ {\it smoothing} if its range consists of smooth vectors. 
Our first main result (Theorem~\ref{thm:smoothop}) 
is a characterization of smoothing operators asserting in particular that 
$A$ is a smoothing operator if and only if the map 
\[ \pi^A \colon G \to B(\cH), \quad g \mapsto \pi(g)A\] 
is smooth with respect to the norm topology on $B(\cH)$. That the smoothness 
of $\pi^A$ implies that $A$ is smoothing is trivial, but the converse is a 
powerful tool whose applications we start to explore in the present paper 
(see Section~\ref{sec:6} for further applications). 
It is amazing that nothing but the metrizability of $G$, resp., of its 
Lie algebra $\g$, is needed in Theorem~\ref{thm:smoothop}. 

To use smoothing operators for the construction of host algebras, 
one needs operators that are naturally constructed from the representation. 
For instance, if $i \colon H \to G$ is a smooth homomorphism, where $H$ 
is a finite dimensional Lie group and 
$\pi_H := \pi \circ i$ has the same smooth vectors as $\pi$, then any operator 
$\pi_H(f) = \int_H f(h)\pi(i(h))\, dh$, $f \in C^\infty_c(H)$, is a natural smoothing 
operator and the $C^*$-algebra $\cA \subeq B(\cH)$ 
generated by $\pi(G)\pi_H(C_c^\infty(H))\pi(G)$ 
is a host algebra whose representations correspond to smooth representations 
of $G$ (Theorem~\ref{thm:5.11}). 

In view of this observation, one would like 
to understand such situations systematically.
In Section~\ref{sec:3} we show that semibounded representations lead to 
an abundance of such situations for $H = \R$.  
More precisely, if 
$x_0 \in\g$ has a neighborhood $U$ such that the operators 
$i\dd\pi(x)$, $x \in U$, are uniformly bounded from above, 
then the one-parameter group $\pi_{x_0}(t) := \pi(\exp tx_0)$ has the 
same smooth vectors as $\pi$ (Theorem~\ref{thmsbsmoothonegen}). 
As this theorem is already far from trivial, 
and new, for finite dimensional Lie groups,\begin{footnote}{In \cite{Ne00b} 
it is proved for irreducible semibounded (=highest weight) representations 
of finite dimensional Lie groups.}  
\end{footnote}
it is quite remarkable that it holds for any infinite dimensional 
Lie group~$G$. The key tool in its proof 
is an application of a variation of Nelson's Commutator Theorem 
(\cite{Nel72}, \cite{RS75}), which can be proved by methods of abstract interpolation theory. Nelson's Commutator Theorem  is an interesting complement 
to Nelson's famous theorem that, if $x_1, \ldots, x_n$ is a basis 
of the Lie algebra $\g$ of a finite dimensional Lie group $G$ and 
$(\pi, \cH)$ is a unitary representation, then 
\[ \cH^\infty = \cD^\infty(\Delta) \quad \mbox{ for } \quad 
\Delta := \oline{\sum_{j = 1}^n \dd\pi(x_j)^2}\] 
(Corollary~9.3 in \cite{Nel69}). In this sense the essentially selfadjoint operator  
$i\dd\pi(x_0)$ plays for a semibounded representation a similar role 
as Nelson's Laplacian $\Delta$ for a representation of a finite dimensional 
Lie group. This situation  clearly demonstrates that, 
although one has very general 
tools that work for all representations of finite dimensional 
Lie groups, different classes of representations of infinite dimensional groups 
require specific but nevertheless equally powerful methods. 

In Section~\ref{sec:4} we show that, if all smooth vectors for 
the unitary representation $\pi_{x_0}$ of $\R$ are smooth for $\pi$, then  
\begin{equation}
  \label{eq:cstara}
\cA := C^*\big(\pi(G)e^{i\oline{\dd\pi}(x_0)}\pi(G)\big) 
\end{equation}
is a host algebra for $G$. From that we derive that, for every subset 
$C \subeq \g'$ in the topological dual space of $\g$ which is weak-$*$ closed, 
convex and $\Ad^*(G)$-invariant and whose support functional 
$s_C(x) := \sup \la C, x \ra$ is bounded on some non-empty open subset of $\g$, 
there exists a natural host algebra 
$\cA_C$ whose representations are precisely those semibounded representations 
of $G$ for which $s_\pi(x) := \sup\Spec(i\oline{\dd\pi}(x)) \leq s_C(x)$ 
holds for every $x \in \g$ (Corollary~\ref{cor:4.9}). Actually, 
these are precisely the $C^*$-algebras in \eqref{eq:cstara} from a different 
perspective. 
This generalizes the corresponding result for the finite dimensional 
case (Section 8 in \cite{Ne08}) 
and for the case where $G$ is a locally convex space 
(Section 7 in \cite{Ne08}). 
The present results on host algebras via smoothing operators 
complement the approach via complex 
semigroups and holomorphic extension of unitary representations 
described in \cite{Ne08}, \cite{Ne00}, \cite{Ze14}, \cite{Ze15} which suffers from the difficulties in constructing 
suitable complex semigroups for infinite dimensional groups. These are 
particularly nasty for the Virasoro group because it has no 
Lie group complexification (cf.~\cite{Ner87}, \cite{Ner90}). 
For positive energy representations of the Virasoro group, an alternative 
approach to the existence of direct integral decompositions 
based on realizations by holomorphic sections has been developed in \cite{NS14}.
We conclude this paper with a brief discussion of criteria for the liminality 
of the constructed host algebras in Section~\ref{sec:5} and 
some remarks on further applications of the present results in 
Section~\ref{sec:6}.

The techniques developed in the present paper have already found several 
applications; some which were quite unexpected. We showcase these applications below.

\subsubsection*{$C^*$-algebras for Lie supergroups} 
For finite dimensional Lie supergroups 
$(G, \g)$, we define in \cite{NS15} a $C^*$-algebra $\cA(G, \g)$ 
whose non-degenerate representations are in one-to-one 
correspondence with unitary representations of $(G,\g)$. However, the methods and arguments of \cite{NS15}
rely heavily on finite dimensionality of $(G, \g)$. 
In a forthcoming article, we are able to go beyond the finite dimensional case by  using smoothing operators and the methods developed in Section \ref{sec:4} for constructing host algebras.
We are able to construct 
universal $C^*$-algebras for certain infinite dimensional 
Lie supergroups such as the central extension of the restricted
 orthosymplectic group and also the 
Neveu--Schwarz and the Ramond supergroup, whose even 
part is the Virasoro group.

\subsubsection*{Automatic regularity of representations} 
In \cite{Ze15b}, 
Theorem 3.4 has been used to obtain criteria for continuous 
representations of a Lie group to be semibounded, hence in 
particular smooth. This applies in particular to positive 
energy representations of so-called oscillator groups 
$G = \Heis(V,\omega) \rtimes_\alpha\R$, where 
$V$ is the space of smooth vectors of a unitary 
one-parameter group $t\mapsto \alpha(t)$ in a complex Hilbert space, and $\omega$ is given by the 
imaginary part of the scalar product. It also applies to 
positive energy representations of double extensions of 
loop groups with compact target group, and the Virasoro group. 

\subsubsection*{Schwartz operators and tracability} 
The concept of smoothing operators has been studied further 
in \cite{DNSZ16}, where we show that, for an operator $S \in B(\cH)$, 
both $S$ and $S^*$ are smoothing if and only if $S$ is 
a {\it Schwartz operator}, which roughly means 
that all operators $\dd\pi(D_1) S \dd\pi(D_2)$, $D_j \in U(\g)$, 
are bounded. In \cite{DNSZ16} we also show that a 
unitary representation $(\pi,\cH)$ of a finite 
dimensional Lie group $G$ is trace class (that is, all operators $\pi(f)$, $f \in C^\infty_c(G)$, are trace class) if and only 
if all smoothing operators are trace class.

\subsubsection*{Host algebras for semibounded representations} 
For double extensions of infinite dimensional loop groups and 
for hermitian Lie groups (corresponding to infinite dimensional 
symmetric Hilbert domains) irreducible semibounded representations 
have been classified in \cite{Ne14} and \cite{Ne12}. The same has been 
achieved for the Virasoro group in \cite{NS14}. 
The construction of host algebras in Section 4 provides natural 
$C^*$-algebras for the semibounded representations of these important 
classes of infinite dimensional Lie groups. 

Not all $C^*$-algebras associated to semibounded representations 
are of type I. In \cite{Ze14} it is  shown that, for certain oscillator groups, 
the semibounded representation theory is not type I, which is inherited 
by the corresponding host algebras. Similar phenomena occur in 
the representation theory of gauge groups \cite{JN15,JN16}, 
where UHF algebras and some of their generalizations, such 
as certain infinite tensor products of the algebra of compact 
operators, arise naturally.

\subsection*{Notation and terminology}

Let $E$ and $F$ be locally convex spaces, $U\subeq E$ open and $f \colon U \to F$ a map. Then the {\it derivative
  of $f$ at $x$ in the direction $h$} is defined as 
\[ \dd f(x)(h) := (\partial_h f)(x) := \frac{d}{dt}\Big|_{t= 0} f(x + t h) 
= \lim_{t \to 0} \frac{1}{t}(f(x+th) -f(x)) \]
whenever it exists. The function $f$ is called {\it differentiable at
  $x$} if $\dd f(x)(h)$ exists for all $h \in E$. 
It is called {\it  continuously differentiable} if it is differentiable at all
points of $U$ and 
$$ \dd f \colon U \times E \to F, \quad (x,h) \mapsto \dd f(x)(h) $$
is a continuous map. Note that this implies that $f$ is continuous and that the maps 
$\dd f(x)$ are linear (cf.\ Theorem~3.2.5 in \cite{Ha82}, 
Lemma~2.2.14 in \cite{GN}). 
The map $f$ is called a {\it $C^k$-map}, $k \in \N \cup \{\infty\}$, 
if it is continuous, the iterated directional derivatives 
$$ \dd^{j}f(x)(h_1,\ldots, h_j)
:= (\partial_{h_j} \cdots \partial_{h_1}f)(x) $$
exist for all integers $1\leq j \leq k$, $x \in U$ and $h_1,\ldots, h_j \in E$, 
and all maps  $\dd^j f \colon U \times E^j \to F$ are continuous. 
As usual, $C^\infty$-maps are called {\it smooth}. 

If $E$ and $F$ are complex locally convex spaces, then $f$ is 
called {\it holomorphic} if it is $C^1$ 
and, for each $x \in U$ the 
map $\dd f(x) \colon E \to F$ is complex linear. 

If $E$ and $F$ are real locally convex spaces, 
then we call a map $f \colon U \to F$, $U \subeq E$ open, 
{\it real analytic} or a $C^\omega$-map, 
if, for each point $x \in U$ there exists an open neighborhood 
$V \subeq E_\C$ and a holomorphic map $f_\C \colon V \to F_\C$ with 
$f_\C\res_{U \cap V} = f\res_{U \cap V}$  (cf.\ \cite{Mil84}). 
Any analytic map is smooth, 
and the corresponding chain rule holds without any condition 
on the underlying spaces, which is the key to the definition of 
analytic manifolds (see \cite{Gl02} for details).

Once the concept of a smooth function 
between open subsets of locally convex spaces is established, it is clear how to define 
a locally convex smooth manifold (cf.\ \cite{Ne06}, \cite{GN}). 
For $r\in\N\cup\{\infty\}$ and $C^r$-manifolds $M$ and $N$, we  write $C^r(M,N)$ 
for the space of $C^r$-maps from $M$ to~$N$.

A {\it (locally convex) Lie group} $G$ is a group equipped with a 
smooth manifold structure modeled on a locally convex space 
for which the group multiplication and the 
inversion are smooth maps. We write $\1$ for the identity element in~$G$. 
Its Lie algebra $\g = \L(G)$ is identified with 
the tangent space $T_\1(G)$. The Lie bracket is obtained by identification with the 
Lie algebra of left invariant vector fields. 
We call $G$ a Banach, resp., a Fr\'echet--Lie group if $\g$ is a Banach, resp., 
a Fr\'echet space. 
Note that the multiplication map of $G$ defines a smooth left action  
$G \times TG \to TG, (g,v) \mapsto g \cdot v$ for which the 
restriction $G \times \g \to TG$ is a diffeomorphism. 
A smooth map $\exp \colon \g \to G$  is called an {\it exponential function} 
if each curve $\gamma_x(t) := \exp(tx)$ is a one-parameter group 
with $\gamma_x'(0)= x$. A Lie group $G$ is said to be 
{\it locally exponential} 
if it has an exponential function which maps an open $0$-neighborhood 
$U$ in $\g$ diffeomorphically onto an 
open subset of $G$. 
If $G$ is a Lie group, then the metrizability of $G$ (as a topological group) 
is equivalent to $G$ being first countable, and this is equivalent to the topology on the 
Lie algebra $\g$ to be defined by a sequence of seminorms. 

Throughout this paper all Lie groups are assumed to have an exponential function. 

\section{Differentiable vectors} 
\label{sec:1}

In this section we refine some of the results in \cite{NS13} concerning 
$C^k$-vectors for continuous representations $\pi \colon G \to \GL(V)$ of a Lie group 
$G$ on a locally convex space~$V$. The main result is Theorem~\ref{thm:2.1} 
which provides a description of the subspace $V^n \subeq V$ of $C^n$-vectors in 
terms of the selfadjoint operators $\oline{\dd\pi}(x)$, $x \in \g$, obtained as infinitesimal 
generators of the one-parameter groups $\pi_x(t) := \pi(\exp tx)$. 

\begin{definition} (a) 
Let $G$ be a Lie group with an exponential function 
and $(\pi, V)$ be a continuous representation on the 
locally convex space $V$. Here continuity means that the corresponding 
action map $G \times V \to V$ is continuous. 
For $v \in V$ we write 
\[ \pi^v \colon G \to V,\quad g \mapsto \pi(g)v \] 
for the orbit map. For $n \in \N_0 \cup \{\infty\}$, we write
\[ V^n := V^n(\pi) := \{ v \in V \colon \pi^v  \in  C^n(G,V)\} \] 
for the subspace of {\it $C^n$-vectors} in $V$. This is a 
$G$-invariant subspace of~$V$.

(b) For every $x \in \g$, we obtain a representation of the additive group $\R$ by 
$\pi_x(t) := \pi(\exp tx)$ and we write 
\[ \oline{\dd\pi}(x) \colon \cD(\oline{\dd\pi}(x)) := V^1(\pi_x) \to V, \quad 
\oline{\dd\pi}(x)v := \frac{d}{dt}\Big|_{t=0} \pi(\exp tx)v \] 
for its {\it infinitesimal generator}. 
Composition of operators defined on subspaces of $V$ is defined  
in the usual way. For $n \in \N_0$, we consider  the subspaces 
\[ \cD^n := \cD^n(\pi) := \bigcap_{x_1, \ldots, x_n \in \g}\cD(\oline{\dd\pi}(x_n)\cdots \oline{\dd\pi}(x_1)) \quad \mbox{ and put } \quad \cD^\infty 
:= \bigcap_{n \in \N} \cD^n.\] For $v \in \cD^n$, we obtain a map 
\[ \omega_v^n \colon \g^n \to V, \quad 
(x_1, \ldots, x_n) \mapsto \oline{\dd\pi}(x_1) \cdots \oline{\dd\pi}(x_n)v.\]
Clearly, $V^n \subeq \cD^n$ and, for every $v \in V^n$, the map $\omega_v^n$ 
is continuous and $n$-linear. Below we shall encounter several 
contexts in which every $v \in \cD^n$ is a $C^n$-vector, at least if we assume that 
the maps $\omega_v^k$ are continuous for $k \leq n$ (cf.~Theorem~\ref{thm:2.1}).  
This equality is crucial if one wants to show that specific elements 
of $V$ are $C^n$-vectors. 

(c) For an operator $A \colon \cD(A) \to V$, we shall use the notation 
\[ \cD^\infty(A) := \bigcap_{n \in \N} \cD(A^n). \] 
If $V$ is a Banach space, we write 
\[ \cD^\omega(A) := \Big\{ v \in\cD^\infty(A)\colon (\exists r > 0)\ 
\sum_{n \geq 0} r^n\frac{\|A^nv\|}{n!} < \infty\Big\}.\] 
\end{definition}

\begin{definition} For a locally convex space $V$ over $\K \in \{\R,\C\}$, we 
denote the topological dual (the space of continuous 
linear functionals $V \to \K$) by $V'$ and endow it with the 
weak-$*$-topology defined by the seminorms 
$\alpha \mapsto |\alpha(v)|$, $v \in V$. Then the natural 
map $V' \into \K^V$ is a topological embedding. 
 
For a representation $(\pi, V)$ of a Lie group on $V$, we obtain a 
natural dual representation $\pi^*(g)\alpha := \alpha \circ \pi(g)^{-1}$ on $V'$. 
In general, this action is not continuous with respect to the weak-$*$-topology, 
but every orbit map $(\pi^*)^\alpha \colon G \to V'$ is continuous, because, 
for every $v \in V$, the map 
\[ \pi^{\alpha,v} \colon G  \to \K, \quad \pi^{\alpha,v}(g) := \alpha(\pi(g)v)  
= (\pi^*)^\alpha(g^{-1})(v)\] 
is continuous. 
We write 
\[ V'_{C^1} := \{ \alpha \in V'\colon (\forall v \in V)\, \pi^{\alpha,v} \in C^1(G,\K) \}\]
for the subspace of $C^1$-vectors for the representation $\pi^*$ on~$V'$.
\end{definition}

\begin{definition} We call a locally convex space $V$ {\it integral complete} 
if, for every continuous function $f \colon [0,1] \to V$, the integral
$\int_0^1 f(t)\, dt$ exists in $V$, i.e., there exists a $w \in V$ such that, 
for every continuous linear functional $\alpha \in V'$ we have 
$\alpha(w) = \int_0^1 \alpha(f(t))\, dt$ 
(cf.\ \cite{Gl12}). 
Sequentially complete spaces are integral complete. 
For a characterization of sequentially
complete spaces in terms of a completeness property, we refer to \cite{We12}.
\end{definition}

\begin{lemma} \label{lem:2.1} 
Assume that the subspace $V'_{C^1} \subeq V'$ of weak-$*$ $C^1$-vectors 
separates the points of $V$ and that $V$ is integral complete. 
Then the following assertions hold: 
\begin{itemize}
\item[\upshape(a)] If $v \in \cD^1$ is such that the map 
$\omega_v \colon \g \to V, x \mapsto \oline{\dd\pi}(x)v$ 
is continuous, then $v \in V^1$. 
\item[\upshape(b)] For $n \in \N$ and $v \in \cD^n$, the map $\omega_v^n$ 
is $n$-linear. 
\end{itemize}
\end{lemma}

\begin{proof} (a) Let $\gamma \colon [-\eps,\eps] \to G$ be a smooth curve
with $\gamma(0) = \1$ and $\gamma'(0) = x$, and 
\[
\xi:[-\eps,\eps]\to\g\ ,\ \xi(t) := \gamma(t)^{-1} \cdot \gamma'(t) \] 
be its left logarithmic derivative. 
For $\alpha \in V'_{C^1}$ and $v \in \cD^1$ we now obtain 
\begin{equation}\label{dalphaw2}
\dd\pi^{\alpha,v}(g \cdot x) 
= \frac{d}{dt}\Big|_{t=0} \pi^{\alpha,v}(g\exp tx) 
= \frac{d}{dt} \Big|_{t=0}\alpha(\pi(g) \pi_x(t)v) 
= \alpha(\pi(g)\oline{\dd\pi}(x)v)
\end{equation}
and thus 
\begin{equation}\label{dalphaw}
\frac{d}{dt} \alpha(\pi(\gamma(t))v) 
= \frac{d}{dt} \pi^{\alpha,v}(\gamma(t)) 
= \dd\pi^{\alpha,v}(\gamma(t)\cdot \xi(t)) 
= \alpha\big(\pi(\gamma(t))\oline{\dd\pi}(\xi(t))v\big).
\end{equation}

Put $\beta(t) := \pi(\gamma(t))v$ for $|t| \leq \eps$. 
The curve 
$\eta(t) := \pi(\gamma(t))\oline{\dd\pi}(\xi(t))v = \pi(\gamma(t)) \omega_v(\xi(t))$ 
is continuous because $\omega_v$ is continuous and the 
action of $G$ on $V$  is continuous. 
By \eqref{dalphaw}, for each $\alpha \in V'_{C^1}$, the function 
$\alpha \circ \beta$ is differentiable and 
$(\alpha \circ \beta)'(t) = \alpha(\eta(t))$. Since  $\alpha \circ \eta$ is continuous, 
it follows that 
\[ \alpha\Big(\int_0^t \eta(\tau)\, d\tau\Big) 
=  \int_0^t \alpha(\eta(\tau))\, d\tau 
= \alpha(\beta(t) - \beta(0)).\] 
As $V'_{C^1}$ separates the points of $V$, we obtain 
$\beta(t) = \beta(0) + \int_0^t \eta(\tau)\, d\tau.$ 
Now the continuity of $\eta$ implies that $\beta$ is $C^1$ with 
$\beta'(0)= \eta(0) = \oline{\dd\pi}(x)v$. 
Finally, Lemma~3.3. in \cite{Ne10b} 
shows that $v$ is a $C^1$-vector. 

(b) Let $v \in \cD^1$. For every $\alpha \in V'_{C^1}$, we have 
$\alpha(\oline{\dd\pi}(x)v) = \dd\pi^{\alpha,v}(\1)x,$
which is linear in $x$. As $V'_{C_1}$ separates the points of $V$, 
it follows that $\omega_v^1 \colon \g \to V$ is linear. 
By induction, we now see that the maps $\omega_v^n \colon \g^n \to V$, 
$v \in \cD^n$, are $n$-linear. 
\end{proof}

The following examples show why the assumption of $G$ to be Fr\'echet is crucial for the 
equality $\cD^k = V^k$. 
\begin{example} We consider the unitary representation 
of the Banach--Lie group $G := (L^p([0,1],\R),+)$, $p \in [1,\infty[$,  
on the Hilbert space $\cH = L^2([0,1],\C)$ by $\pi(g)f := e^{ig}f$. 
In Section~10 of \cite{Ne10b} we have seen that this representation 
is continuous with 
\begin{equation}
  \label{eq:dk-ex}
 \cD^{k} = \begin{cases}
\{0\} \text{ for } k > \frac{p}{2} \\
L^\infty([0,1])  \text{ for } k = \frac{p}{2} \\
L^{\frac{2p}{p-2k}}([0,1]) \text{ for } k < \frac{p}{2}. 
\end{cases} 
\end{equation}
In particular, $\cD^1$ is dense for $p \geq 2$. 
As $G$ is a Fr\'echet space, it follows from 
Theorem~6.3 in \cite{NS13} that, for $p \geq 2$, 
we have $\cH^k = \cD^k$ for $k \in \N$. 

Now let $q > p$ and consider the subgroup $H := L^q([0,1],\R)$ of $G$ which is a 
Lie group in its own right with respect to the subspace topology. As this subspace is not 
closed, $H$ is not a Fr\'echet--Lie group. The Lie algebra $\fh$ of $H$ can be identified 
with the subspace $L^q([0,1],\R)$ of $\g = L^p([0,1],\R)$. For this Lie algebra, 
we find with \eqref{eq:dk-ex} 
\[ \cD^{k}(\fh) = \begin{cases}
\{0\} \text{ for } k > \frac{q}{2} \\
L^\infty([0,1])  \text{ for } k = \frac{q}{2} \\
L^{\frac{2q}{q-2k}}([0,1]) \text{ for } k <\frac{q}{2}. 
\end{cases} \]
For $k < q/2$ it follows in particular that 
$\cD^k(\fh)$ is strictly larger than $\cD^k = \cD^k(\g)$. However, the subspace of those elements 
$v \in \cD^k(\fh)$ for which the $n$-linear maps $\omega_v^n \colon \fh^n \to \cH$, $n \leq k$, are 
continuous coincides with the smaller space $\cD^k(\g) = \cH^k$. In particular, $H$ and $G$ 
have the same spaces of $C^k$-vectors. 
\end{example}

The following theorem generalizes the criterion 
for $C^k$-vectors for locally exponential Lie groups 
given in Lemma~3.4 in \cite{Ne10b} to general Lie groups and 
part (ii) generalizes Theorem~6.3 in \cite{NS13} to non-unitary representations. 

\begin{theorem} \label{thm:2.1} 
Let $(\pi, V)$ be a continuous representation of the Lie group $G$ with exponential 
map on the integral complete locally convex space $V$ such that 
the subspace $V'_{C^1}\subeq V'$ of weak-$*$-$C^1$ 
vectors for the adjoint representation on $V'$ separates the points of~$V$. 
Then the following assertions hold: 
\begin{itemize}
\item[\upshape(i)]  For  $k \in \N$, a vector 
$v\in V$ is a $C^k$-vector if and only if $v\in \cD^k$ and the maps
$\omega_v^n$, $n \leq k$, are continuous. In particular, $v\in V^\infty$ 
if and only if $v\in\cD^\infty$ and, for every $n \in \N$, the map 
$\omega_v^n$ is continuous.
\item[\upshape(ii)] If $\g$ is metrizable and Baire, e.g., Fr\'echet, 
and $V$ is metrizable, then 
$\cD^n=V^n$ for every $n\in\N\cup\{\infty\}$, i.e., every element of 
$\cD^n$ is a $C^n$-vector.
\end{itemize}
\end{theorem}

\begin{proof} (i) If $v$ is a $C^k$-vector, then $v\in \cD^k$ 
and the maps $\omega_v^n,n\leq k$, are continuous and $n$-linear 
because they arise as partial derivatives of the orbit map $\pi^v$
(cf.~Remark~3.2 in \cite{Ne10b}). 

Conversely, let $v\in \cD^k$ such that the maps $\omega_v^n,n\leq k$, 
are continuous. From Lemma~\ref{lem:2.1}(b) we know that $\omega_v^n$ is $n$-linear. 
For $k = 1$, Lemma~\ref{lem:2.1}(a) implies that $v$ is a $C^1$-vector. 
Assume that  $k>1$. 
Since $v\in V^1$, the differential of the orbit 
map $\pi^v(g) = \pi(g)v$ is a continuous map 
\[T(\pi^v):T(G)\simeq G\times\g\to V\ ,\ (g,x)\mapsto T(\pi^v)(g,x) = \pi(g)\oline{\dd\pi}(x)v.\]
It remains to prove that $T(\pi^v)$ is $C^{k-1}$. 
As $T(\pi^v)(g,x)=\pi^w(g)$ for 
$w:=\oline{\dd\pi}(x)v\in\cD^{k-1}$, our induction hypothesis implies 
that $w\in V^{k-1}$. Thus $T(\pi^v)$ has directional derivatives of order 
$j\leq k-1$ and they are sums of terms of the form 
$\pi(g)\omega_w^j(x_1,\ldots,x_j)$ 
for $j\leq k-1$. As $w$ is a $C^{k-1}$-vector, all these maps 
are continuous (Remark~3.2(a) in \cite{Ne10b}). We conclude that $T(\pi^v)$ is a 
$C^{k-1}$ map, i.e., $\pi^v$ is~$C^k$. 

(ii) Suppose that $V$ is metrizable and that $\g$ is metrizable and Baire. 
We argue by induction on $n \in \N_0$ that the maps $\omega_v^n$, 
$v \in \cD^n$, are continuous. 
For $n = 0$, the constant map $\omega_v^0 = v$ is continuous. 
Assume $n>0$ and that $\omega_w^{n-1} \colon \g^{n-1} \to V$ 
is a continuous $(n-1)$-linear map for every $w \in \cD^{n-1}$. 

Hence, for $t>0$, the maps $F_t \colon  \g^n \to V$, defined by 
\[ F_t(x_1, x_2,\ldots, x_n) := 
\frac{1}{t}\Big(\pi(\exp (t x_1)) \omega_v^{n-1}(x_2,\ldots, x_n)- 
\omega_v^{n-1}(x_2,\ldots, x_n)\Big) \] 
are continuous and satisfy 
\[ \lim_{ n \to \infty} F_{\frac{1}{n}}(x_1, x_2,\ldots, x_n) 
= \omega_v^n(x_1,\ldots, x_n).\] 
Fix $y_k \in\g$ for $k \not=j$ and consider the linear map 
\[ f \colon \g \to V, \quad f(x) := \omega_v^n(y_1, \ldots, y_{j-1}, x, y_{j+1}, 
\ldots, y_n).\] 
Since $V$ is metrizable and $\g$ is a Baire space, it 
follows from Example~22(a) in Ch.~IX, \S 5 of \cite{Bou74} that 
the set of discontinuity points of $f$ is of the first category, 
hence not all of $\g$. We conclude that there exists a point in which 
$f$ is continuous, so that its continuity follows from linearity. 
As $\g$ is Baire metrizable, the continuity of $\omega_v^n$ 
follows from Corollary~1.6 in \cite{BS71}. 
Now (ii) follows from~(i). 
\end{proof}

The following corollary is a tool to identify the smooth vectors for the 
left multiplication action on $B(\cH)$ defined by a unitary representation of~$G$ 
(cf.\ Theorem~\ref{thm:smoothop} below).

\begin{corollary} \label{cor:2.1} 
Let $(\rho, \cH)$ be a unitary representation of $G$ 
for which the subspace $\cH^1$ of $C^1$-vectors is dense. 
For the following representations we have that, for  $k \in \N$, a vector 
$v\in \cH$ is a $C^k$-vector if and only if $v\in \cD^k$ and the maps
$\omega_v^n$ are continuous for $n\leq k$: 
\begin{itemize}
\item[\upshape(i)] The representation 
of $G$ on $V := \{\, T \in B(\cH) \colon g \mapsto \rho(g)T \ \mbox{ is continuous}\, \}$ 
given by left multiplication $\pi(g)(T) := \rho(g)T$. 
\item[\upshape(ii)] The representation 
$\alpha(g,h)(A) := \rho(g)A\rho(h)^{-1}$ of $G \times G$ on the subspace 
$V \subeq B(\cH)$ of continuous vectors for this 
representation. 
\end{itemize}
\end{corollary}

\begin{proof} In both cases the linear functionals 
$\alpha_{v,w}(A) := \la Av, w\ra$, where $v,w \in \cH$ are $C^1$-vectors,  
are contained in $B(\cH)_{C^1}'$ and separate the points. 
Hence the assertion follows from Theorem~\ref{thm:2.1}(i). 
\end{proof}

\section{Smoothing operators for unitary representations}
\label{sec:2}

In this section we introduce the core concept of this paper: 
\begin{definition} For a unitary representation $(\pi, \cH)$ of a Lie group $G$ on $\cH$, 
an operator $A \in B(\cH)$ is called a {\it smoothing operator} if 
$A(\cH) \subeq \cH^\infty$. \end{definition}

It will turn out that this class of operators is extremely useful when it comes 
to constructing host algebras for $G$ (Section~\ref{sec:4}). 
The main result of this section is the Characterization Theorem~\ref{thm:smoothop}. 
Its main point is the fact that, if $G$ is metrizable, then  
$A$ is a smoothing operator if and only if the map 
\[ \pi^A \colon G \to B(\cH), \quad g \mapsto \pi(g)A\] 
is smooth with respect to the norm topology on $B(\cH)$. 
We start with a brief discussion of various types of examples.

\subsection{Examples of smoothing operators} 

\begin{example} \label{ex:3.1} Let $(\pi,\cH)$ be a continuous unitary representation of $G$. 

(a) If $G$ is finite dimensional and 
$C^\infty_c(G)$ is the space of test functions on $G$, then, for every $f \in 
C^\infty_c(G)$,
 the operator $\pi(f) = \int_G f(g)\pi(g)\, dg$ is a smoothing operator. 

Since the integrated representation $\pi \colon L^1(G) \to B(\cH)$ is $G$-equivariant, 
we see that, more generally, for every $f \in L^1(G)$, which is a smooth vector for the 
left translation 
action of $G$ on $L^1(G)$, the operator $\pi(f)$ is smoothing. 
Since the left regular representation $\lambda$ of $G$ on $L^1(G)$ is a continuous Banach 
representation, we have $L^1(G)^\infty = \cD^\infty$ (Theorem~1.3 
in \cite{Ne10b}). 
Integrating against test functions further shows that 
\[ \oline{\dd\lambda}(x)f = - L_X f \quad \mbox{ with} \quad 
(L_X f)(g) = \dd f(g)(X.g) \] 
in the sense of distributions.  
From this discussion one can derive that $f \in L^1(G)$ is a smooth vector 
if and only if 
$f \in C^\infty(G)$ with $D * f \in L^1(G)$ for every 
$D \in U(\g)$ (see the proof of Proposition~3.27 in \cite{Mag92}, where a similar 
result for $L^2(G)$ is derived from Sobolev's Lemma). 

(b) Suppose that $x \in \g$ is such that the operator $\oline{\dd\pi}(x)$ has the same smooth 
vectors as $G$, i.e., $\cH^\infty = \cD^\infty(\oline{\dd\pi}(x)) = \cH^\infty(\pi_x)$ for 
$\pi_x(t) = \pi(\exp tx)$. Then (a) implies that, 
for every $f \in C^\infty(\R)$ for which all derivatives $f^{(k)}$, $k \in \N$, are 
integrable, the operator 
\[ \pi_x(f) := \int_\R f(t)\pi(\exp tx)\, dt \] 
is a smoothing operator. This holds in particular for $f \in \cS(\R)$. 
In view of the Spectral Theorem for selfadjoint operators, it even suffices that 
all functions $x \mapsto x^n\hat f(x)$ are bounded, i.e., that $\hat f$ vanishes 
rapidly at infinity. 

(c) Assume that $G$ is finite dimensional and let  
$x_1, \ldots, x_n$ be a basis of its Lie algebra $\g$. 
Then 
\[ \cH^\infty = \cD^\infty(\Delta) \quad \mbox{ for } \quad 
\Delta := \oline{\sum_{j = 1}^n \dd\pi(x_j)^2}\] 
(Corollary~9.3 in \cite{Nel69}) and since $\Delta$ has non-positive spectrum, 
it follows that the contraction semigroup $(e^{t\Delta})_{t > 0}$, which is an 
abstract version of the heat semigroup on $L^2(G)$, consists 
of smoothing operators. 
\end{example}

\begin{example} \label{ex:holext} (Smoothing operators by holomorphic extension) 
Suppose that the Lie group $G$ sits in a complex Lie group $G_\C$ and that 
$S \subeq G_\C$ is an open subsemigroup satisfying $G S \subeq S$. 

We assume that $(\pi, \cH)$ is a unitary representation for which 
there exists a holomorphic representation $\hat\pi \colon S \to B(\cH)$ such that 
\[ \hat\pi(gs) = \pi(g) \hat\pi(s) \quad \mbox{ for }\quad g \in G, s \in S \] 
(see \cite{Ne00} and \cite{Ne08} for a detailed discussion of such situations). 
For every $s \in S$, the map $G \to B(\cH), g \mapsto \pi(g)\hat\pi(s) = \hat\pi(gs)$ 
is smooth because $\hat\pi$ is holomorphic. 
Therefore $\pi(S)$ consists of smoothing operators. 
\end{example}

\begin{proposition} Let $(\pi, \cH)$ be a unitary representation, 
$\iota_H \colon H \to G$ a morphism of Lie groups and $\pi_H := \pi\circ \iota_H$. Assume that 
$H$ is finite dimensional. 
Then the following are equivalent: 
\begin{itemize}
\item[\upshape(i)] $\cH^\infty = \cH^\infty(\pi_H)$. 
\item[\upshape(ii)] All operators $\pi_H(f)$, $f \in C^\infty_c(H)$, are smoothing. 
\end{itemize}
\end{proposition}

\begin{proof} (i) $\Rarrow$ (ii) follows from Example~\ref{ex:3.1}(a). 

(ii) $\Rarrow$ (i):  If all operators $\pi_H(f)$, $f \in C^\infty_c(H)$, 
are smoothing, then the Dixmier--Malliavin Theorem (\cite{DM78}) shows that 
$\cH^\infty(\pi_H) = \pi_H(C^\infty_c(H))\cH \subeq \cH^\infty \subeq \cH^\infty(\pi_H).$
\end{proof}

\begin{corollary} Let $(\pi, \cH)$ be a unitary representation, $x \in \g$ 
and $\pi_x(t) := \pi(\exp tx)$. Then the following are equivalent: 
\begin{itemize}
\item[\upshape(i)] $\cH^\infty = \cD^\infty(\oline{\dd\pi(x)})$. 
\item[\upshape(ii)] All operators $\pi_x(f)$, $f \in C^\infty_c(\R)$, are smoothing. 
\end{itemize}
\end{corollary}

\begin{proposition} Let $(\pi, \cH)$ be a unitary representation, $x \in \g$ and 
$\pi_x(t) := \pi(\exp tx) = e^{t\oline{\dd\pi}(x)}$. Suppose that 
$i\oline{\dd\pi}(x)$ is bounded from above. 
Then all operators $(e^{it \oline{\dd\pi}(x)})_{t> 0}$ are smoothing if and only if 
$\cD^\omega(\oline{\dd\pi}(x)) \subeq \cH^\infty$. 
\end{proposition}

\begin{proof} This follows from the fact that the 
subspace $\bigcup_{t > 0} e^{it \oline{\dd\pi}(x)}\cH$ coincides with 
the space $\cD^\omega(\oline{\dd\pi}(x))$ of analytic vectors. 
\end{proof}

\begin{example} Let $H$ be a finite dimensional Lie group and $M$ a homogeneous 
space of $H$. Then $M$ carries a smooth nowhere vanishing $1$-density which leads to a  
$\Diff(M)$-quasiinvariant measure $\mu$ on $M$. 
We assume that $M$ is compact, so that $\Diff(M)$ is a Fr\'echet Lie group.
We thus obtain a 
unitary representation $(\pi, \cH)$ of $G := \Diff(M)$ on $\cH := L^2(M,\mu)$. 
By the Dixmier--Malliavin Theorem (\cite{DM78}), the subspace of smooth vectors 
for $\pi_H$ is generated by the images of the operators 
$\pi_H(f)$, $f\in C^\infty_c(H)$, and since $H$ acts transitively on $M$, 
this implies that $\cH^\infty(\pi_H) \subeq C^\infty(M)$ which easily 
implies that $C^\infty(M) = \cH^\infty = \cH^\infty(\pi_H)$. 
\end{example}

\subsection{A characterization of smoothing operators} 

\begin{lemma} \label{lem:prod} 
Let $\cH$ be a complex Hilbert space, $A \in B(\cH)$ and $B \colon \cD(B) \to \cH$ 
be a densely defined operator on $\cH$. Then the following assertions hold: 
\begin{itemize}
\item[\upshape(a)] The densely defined operator $AB \colon \cD(B) \to \cH$ is bounded if and only if 
$A^*(\cH) \subeq \cD(B^*)$. If this is the case, then $(AB)^* = B^*A^*$. 
\item[\upshape(b)] If $B^*$ is also densely defined, then $A^*B^*$ is bounded if and only 
if $A(\cH) \subeq \cD(\oline B)$. If this is the case, then $\oline BA$ is also bounded. 
\end{itemize}
\end{lemma}

\begin{proof} (a) Suppose first that $AB$ is bounded. Then, for $v \in \cH$, the linear 
functional 
\[ \cD(B) \to \C, \quad w \mapsto \la ABw, v \ra = \la Bw, A^* v \ra \] 
is continuous. This implies that $A^*v \in \cD(B^*)$ and $B^*A^*v = (AB)^*v$. 

If, conversely, $A^*\cH \subeq \cD(B^*)$, then 
$\cD(B^*A^*) = \cH$. Since we also have $B^*A^* \subeq (AB)^*$, we obtain equality. 
Hence $(AB)^*$ is an everywhere defined closed operator and therefore bounded. This in turn shows that 
$AB \subeq \oline{AB} = ((AB)^*)^*$ is bounded. 

(b) In view of (a), $A^*B^*$ is bounded if and only if 
$A(\cH) \subeq \cD(B^{**}) = \cD(\oline B)$. Then 
$\oline B A = B^{**}A = (A^* B^*)^*$ is closed and everywhere defined, hence bounded. 
\end{proof}

\begin{lemma}\label{lem:der-rep} 
Let $(\pi, \cH)$ be a continuous unitary representation of the Lie group $G$ with 
exponential function. Let $B(\cH)_c \subeq B(\cH)$ denote the space 
of continuous vectors for the left multiplication action on $B(\cH)$ 
and write  $\lambda(g)(A) := \pi(g)\circ A$ for the corresponding continuous representation 
of $G$ on $B(\cH)_c$. Fix  $x \in \g$.
\begin{itemize}
\item[{\upshape (i)}]
Let $A\in B(\cH)_c$. If $A\in \cD(\oline{\dd\lambda}(x))$, then  
$A(\cH) \subeq \cD(\oline{\dd\pi}(x))$ and $ 
\oline{\dd\lambda}(x)A = \oline{\dd\pi}(x) A$. 
\item[{\upshape (ii)}] Let $A\in B(\cH)_c$ such that
the operator  $\oline{\dd\pi}(x)^2 A$ is defined on $\cH$. Then 
$\oline{\dd\pi}(x)^2 A$ is bounded
and $A \in \cD(\oline{\dd\lambda}(x))$. 
\end{itemize}
\end{lemma} 

\begin{proof} Since the operators $\pi(g)$ are unitary, 
$B(\cH)_c$ is a norm closed subspace of $B(\cH)$ on which 
$\lambda$ defines a continuous representation of $G$ by isometries. 

(i)
For $x \in \g$, $v \in \cH$, $w \in \cD(\oline{\dd\pi}(x))$ 
and $A \in \cD(\oline{\dd\lambda}(x))$, 
we have 
\[ \la \oline{\dd\lambda}(x)Av, w \ra = \frac{d}{dt}\Big|_{t=0} \la Av, \pi(\exp(-tx))w \ra 
= -\la Av, \oline{\dd\pi}(x)w \ra,\] 
so that $A\cH \subeq \cD(\oline{\dd\pi}(x)^*) = \cD(\oline{\dd\pi}(x))$ and 
\[ \oline{\dd\lambda}(x) A = -\oline{\dd\pi}(x)^* \circ A = \oline{\dd\pi}(x) \circ A.\] 

(ii)
 From $A\cH \subeq  \cD(\oline{\dd\pi}(x))$
and Lemma~\ref{lem:prod}(b) it follows that 
$B := \oline{\dd\pi}(x)  A$ is bounded.
Now let $S\in B(\cH)$ be a bounded operator satisfying $S\cH \subeq  \cD(\oline{\dd\pi}(x))$. 
By Lemma~\ref{lem:prod}(b), the operator
$T := \oline{\dd\pi}(x)  S$ is bounded.
For every $v,w\in\cH$ such that $\|v\|=\|w\|=1$,
\begin{eqnarray}
\label{int-hadi-weak}
\frac{1}{t}\langle ( \pi(\exp tx) S - S)v,w\rangle - \langle Tv,w\rangle
&=& 
\frac{1}{t} \int_0^t \langle \pi(\exp s x)Tv - Tv,w\rangle ds \notag\\
&=& 
\frac{1}{t} \int_0^t \langle Tv,\pi(\exp(-sx))w-w\rangle ds.
\end{eqnarray}
From \eqref{int-hadi-weak} it follows that
$\|\pi(\exp(tx))S-S\|\leq 3|t|\cdot\|T\|$, so that $\lim_{t\to 0}\|\pi(\exp(tx))S-S\|=0$.
Setting $S:=B=\oline{\dd\pi}(x)A$ and $T:=\oline{\dd\pi}(x)^2A$ in \eqref{int-hadi-weak}, we obtain 
\begin{equation}
\label{hadi-limB}
\lim_{t \to 0} \|\pi(\exp t x)B -B\| = 0.
\end{equation}
Next, setting $S:=A$ and $T:=B$ 
in \eqref{int-hadi-weak}
and using \eqref{hadi-limB}, we obtain
that the equality
\[ \lim_{t \to 0} \frac{1}{t}\big( \pi(\exp tx) A - A\big) = B \] 
holds in the norm topology on $B(\cH)$. Since $B(\cH)_c$ is closed in $B(\cH)$, we obtain that
$B \in B(\cH)_c$ and  $A \in \cD(\oline{\dd\lambda}(x))$. 
\end{proof}

\begin{lemma}\label{lem:der-rep2} 
Let $(\pi, \cH)$ be a continuous unitary representation of the Lie group $G$ with 
exponential function. Let $B(\cH)_c \subeq B(\cH)$ denote the space 
of continuous vectors for the right  multiplication action on $B(\cH)$ 
and write  $\rho(g)A := A\pi(g)^{-1}$ for the corresponding continuous representation 
of $G$ on $B(\cH)_c$. Fix $x \in \g$. 
\begin{itemize}
\item[{\upshape (i)}]
Let $A\in B(\cH)_c$ such that $A \in \cD(\oline{\dd\rho}(x))$. Then 
$\|A \oline{\dd\pi}(x)\| < \infty$ and  $ 
\oline{\dd\rho}(x)A = -A\oline{\dd\pi}(x)$.
\item[{\upshape (ii)}] Let $A\in B(\cH)_c$ such that the operator 
$A\oline{\dd\pi}(x)^2$ is bounded. Then 
$A \in \cD(\oline{\dd\rho}(x))$. 
\end{itemize}
\end{lemma} 

\begin{proof} 
(i) 
Clearly, 
$A \in \cD(\oline{\dd\rho}(x))$ if and only if 
$A^* \in \cD(\oline{\dd\lambda}(x))$. If this is the case, then 
Lemma~\ref{lem:der-rep} implies that 
$\oline{\dd\lambda}(x)A^* = \oline{\dd\pi}(x)A^*$. 
Then $A\oline{\dd\pi}(x)$ is bounded by Lemma~\ref{lem:prod} and we obtain 
\[ \oline{\dd\rho}(x)A 
= \big(\oline{\dd\lambda}(x)A^*\big)^* 
= \big(\oline{\dd\pi}(x)A^*\big)^* = -A \oline{\dd\pi}(x).\] 

(ii) If $A\oline{\dd\pi}(x)^2$ is bounded, then 
$\oline{\dd\pi}(x)^2 A^*$ is bounded and defined on $\cH$ 
(Lemma~\ref{lem:prod}), so that $A^* \in \cD(\oline{\dd\lambda}(x))$ 
by Lemma~\ref{lem:der-rep}, and hence 
$A \in \cD(\oline{\dd\rho}(x))$. 
\end{proof}

The following characterization of smooth vectors 
is an extremely powerful tool. As we shall see in Section~\ref{sec:4} below, 
it can be used to construct $C^*$-algebras for infinite dimensional Lie groups 
and we expect a variety of other applications (cf.\ Section~\ref{sec:6}). 

\begin{theorem}[Characterization Theorem for smoothing operators] 
\label{thm:smoothop} 
Let $(\pi, \cH)$ be a smooth unitary representation of a Lie group~$G$ with exponential 
function and $A \in B(\cH)$. Consider the assertions 
\begin{itemize}
\item[\upshape(i)] The map $G \to B(\cH), g \mapsto \pi(g) A$ is smooth. 
\item[\upshape(ii)] The map $G \to B(\cH), g \mapsto A^*\pi(g)$ is smooth. 
\item[\upshape(iii)] $A\cH \subeq \cH^\infty$. 
\item[\upshape(iv)] $A\cH \subeq \cD^\infty$. 
\item[\upshape(v)] All operators $A^* \dd\pi(D)$, $D \in U(\g_\C)$, are bounded on $\cH^\infty$. 
\end{itemize}
Then we have the implications 
{\upshape(i) $\Leftrightarrow$ (ii) $\Rarrow$ (iii) $\Rarrow$ (iv) $\Leftrightarrow$ (v)}. 
If $\g$ is metrizable, then {\upshape(iii) $\Rarrow$ (i),(ii)} and if, in addition, 
$\g$ is metrizable Baire, e.g., Fr\'echet, then {\upshape(i)-(v)} are equivalent.
Condition {\upshape(iii)} implies that all operators 
$\dd\pi(D) A,  D\in U(\g_\C),$ are bounded. 
\end{theorem}

\begin{proof} (i) $\Leftrightarrow$ (ii) follows from the fact that the map 
$B \mapsto B^*, B(\cH) \to B(\cH)$ is real linear and continuous, 
hence smooth. 

(i) $\Rarrow$ (iii): For every $v \in \cH$, the map $G \to \cH, g \mapsto \pi(g)Av$ 
is smooth 
by (i), so that $Av \in \cH^\infty$. 

(iii) $\Rarrow$ (iv) follows from $\cH^\infty \subeq \cD^\infty$. 

(iv) $\Rarrow$ (v): For $x_1, \ldots, x_n \in \g$, (iv) implies that 
\begin{align*}
A\cH \subeq \cD^\infty 
&\subeq \cD(\oline{\dd\pi}(x_1) \cdots \oline{\dd\pi}(x_n)) 
= \cD(\dd\pi(x_1)^* \cdots \dd\pi(x_n)^*)\\
& \subeq  \cD((\dd\pi(x_n) \cdots \dd\pi(x_1))^*).
\end{align*}
Lemma~\ref{lem:prod}(a) therefore shows that 
$A^*\dd\pi(x_n) \cdots \dd\pi(x_1)$ is bounded. This implies~(v).

We also obtain from Lemma~\ref{lem:prod}(a) that 
the operators 
$(\dd\pi(x_n) \cdots \dd\pi(x_1))^* A = 
(-1)^n \dd\pi(x_1) \cdots \dd\pi(x_n) A$ are bounded, which is the last assertion 
of the theorem. 

(v) $\Rarrow$ (iv): Let $v \in \cH$ and $D \in U(\g_\C)$. Since  $A^* \dd\pi(D) \colon \cH^\infty \to \cH$ 
is bounded, Lemma~\ref{lem:prod}(b) implies that $Av \in \cD(\dd\pi(D)^*)$. 
For every $x \in \g$, the operator $\dd\pi(x)^* = -\oline{\dd\pi}(x)$ is the infinitesimal 
generator of the corresponding unitary one-parameter group. We thus obtain 
by Lemma~\ref{lem:prod}(a) that $Av \in \cD(\dd\pi(x)^*) = \cD(\oline{\dd\pi}(x))$ 
with $\dd\pi(x)^*A = (A^* \dd\pi(x))^*$ and in particular $Av \in \cD^1$. 

Suppose, by induction, that 
\[  A\cH \subeq \cD^n = \bigcap_{x_1, \ldots, x_n \in \g} 
\cD(\dd\pi(x_n)^* \cdots \dd\pi(x_1)^*)\] 
and 
\[ (A^*\dd\pi(x_1)\cdots\dd\pi(x_n))^*
=
\dd\pi(x_n)^*\cdots \dd\pi(x_1)^*A.\] 
For $x_1, \ldots, x_{n+1}\in \g$, the boundedness of 
$\big(A^* \dd\pi(x_1) \cdots \dd\pi(x_n)\big) \dd\pi(x_{n+1}),$ 
leads with Lemma~\ref{lem:prod}(a) to 
\[ \big(A^* \dd\pi(x_1) \cdots \dd\pi(x_n)\big)^* v \in \cD(\dd\pi(x_{n+1})^*)\] 
and 
\[ (A^*\dd\pi(x_1)\cdots\dd\pi(x_{n+1}))^*
=
\dd\pi(x_{n+1})^*\cdots \dd\pi(x_1)^*A.\]  
In particular,  $Av \in \cD^{n}$ for every $n\in\N$. This shows that $Av \in \cD^\infty$.

(iv) $\Leftrightarrow$ (iii): If, in addition, $\g$ is metrizable Baire, 
then Theorem~\ref{thm:2.1}(ii) implies that $\cD^\infty = \cH^\infty$. 

(iii) $\Rarrow$ (i): We now assume that $\g$ is metrizable. 
Let $A \in B(\cH)$ be a smoothing operator and $n \in \N$. We have 
already seen above that, for $x_1, \ldots, x_n \in \g$, the operator 
$\dd\pi(x_1) \cdots \dd\pi(x_n)A$ is bounded. Next we observe that the $n$-linear map 
\[ F \colon\g^n \to B(\cH), \quad F(x_1, \ldots, x_n) := \dd\pi(x_1) \cdots \dd\pi(x_n)A \] 
has the property that, for every $v \in \cH$, the map 
\[ F^v \colon\g^n \to\cH, \quad F^v(x_1, \ldots, x_n) := \dd\pi(x_1) \cdots \dd\pi(x_n)Av \] 
is continuous because $Av \in \cH^\infty$. Since $\g$, and therefore $\g^n$, is metrizable, Proposition~5.1 in \cite{Ne10b} implies that $F$ is continuous. 

In view of Corollary~\ref{cor:2.1}(i), we have to show 
for the left multiplication action $\lambda \colon G \to B(B(\cH))$ that 
$A \in B(\cH)_c$, that 
$A \in \cD^n(\lambda)$ for every $n$, and that the corresponding 
$n$-linear maps 
$\omega_A^n \colon \g^n \to B(\cH)$ are continuous. 
First we assume that we have shown that
$\dd\pi(x_1)\cdots \dd\pi(x_n)A \in B(\cH)_c$ for every $n\in\N_0$.
From Lemma~\ref{lem:der-rep} we know that, 
for $x \in \g$, the domain of $\oline{\dd\lambda}(x)$ 
contains all operators $T \in B(\cH)_c$ for which 
$\oline{\dd\pi}(x)^2T$ is defined on $\cH$ and in this case 
$\oline{\dd\lambda}(x)T = \oline{\dd\pi}(x)T$. 
This readily implies that,
for $x_1, \ldots, x_n \in \g$, the operator 
$A$ is contained in $\cD^\infty(\lambda)$ with 
\[ \dd\pi(x_1) \cdots \dd\pi(x_n) A 
= \oline{\dd\lambda}(x_1) \cdots \oline{\dd\lambda}(x_n) A,\] 
so that the map $F$ above coincides with $\omega_A^n$, and thus $\omega_A^n$ is continuous. 

It now remains to show that (iii) implies $\dd\pi(x_1)\cdots \dd\pi(x_n)A \in B(\cH)_c$. 
First we prove the latter statement for $n=0$, i.e., that the map 
$G \to B(\cH), g \mapsto \pi(g)A$ is continuous. The same argument then applies 
to all operators $\dd\pi(x_1) \cdots \dd\pi(x_n) A$. 
We have already seen that 
\[ \omega_{A} \colon \g \to B(\cH), \quad 
x \mapsto \dd\pi(x) \circ A \] 
is continuous. 
Consider a smooth curve $\gamma \colon [0,1] \to G$ starting in $\1$ and ending in 
$g$. Let $\delta(\gamma) \colon [0,1] \to \g$ denote the left logarithmic derivative of $\gamma$, i.e., 
$\gamma'(t) = \gamma(t) \delta(\gamma)_t$. For $v \in \cH$, we then have 
\[ \pi(\gamma(1))Av - Av = \int_0^1 \pi(\gamma(t))\dd\pi(\delta(\gamma)_t) A v\, dt,\] 
which leads to the estimate 
\[ \|\pi(\gamma(1))A-A\| \leq \sup \{ \|\omega_{A}(\delta(\gamma)_t)\|\colon 0 \leq t \leq 1\}.\] 

Let $U_\g \subeq \g$ be a convex 
open $0$-neighborhood and $\phi \colon U_\g \to G$ be a chart of $G$ 
with $\phi(0) = \1$. If $\theta_G(v_g) := g^{-1} \cdot v_g$ denotes the left Maurer--Cartan 
form of $G$, then $\theta := \phi^*\theta_G \in \Omega^1(U_\g,\g)$ is a smooth $1$-form, 
hence a smooth function on $U_\g \times \g$. 
For $x \in U_\g$, we obtain a smooth curve 
$\gamma_x(t) := \phi(tx)$ in $G$ with $\gamma_x(1) = \phi(x)$. 
Next we observe that, for $\eta_x(t) = tx$, the map 
\[ H \colon U_\g \to C([0,1],\g), \quad H(x)(t) := \delta(\gamma_x)_t 
= (\eta_x^*\theta)(t) = \theta(tx)(x) \] 
is a continuous function because the function $U_\g \times [0,1] \to \g, (x,t) \mapsto \theta(tx)(x)$ 
is continuous. We thus obtain in particular $\lim_{x \to 0} \|H(x)\|_\infty = 0$, and this implies 
\[ \|\pi(\phi(x))A-A\| = \|\pi(\gamma_x(1))A-A\| 
\leq \sup \{ \|\omega_{A}\big(H(x)(t)\big)\|\colon 0 \leq t \leq 1\} \to 0.\] 
for $x \to 0$. This completes the proof of (i). 
\end{proof}

\section{Smooth vectors for semibounded representations} 
\label{sec:3}

In this section we take a closer look at the subspace of 
smooth vectors for a semibounded unitary representation 
$(\pi, \cH)$. Our main result is the surprisingly general fact that, 
if $x_0 \in\g$ has a neighborhood $U$ such that the operators 
$i\dd\pi(x)$, $x \in U$, are uniformly bounded from above, 
then the one-parameter group $\pi_{x_0}(t) := \pi(\exp tx_0)$ has the 
same smooth vectors as $\pi$ (Theorem~\ref{thmsbsmoothonegen}). 

\subsection{Scales of Hilbert spaces}
\label{subsec:3.1}

Before we turn to the proof of the main theorem of this section, we recall 
the setting of Nelson's Commutator Theorem (Section~X.5 in \cite{RS75}). 

Let $N\geq\1$ be a self-adjoint operator on the Hilbert space $\cH$. For $k\in\Z$, 
we define $\cH_k$ as the completion of the dense subspace $\cD(N^{k/2})$ of $\cH$ with 
respect to the norm 
\[ \|v\|_k:= \|N^{k/2}v\|.\]
For $k \geq 0$, $\cD(N^{k/2})$ is complete with respect to $\|\cdot\|_k$, so that 
$\cH_k = \cD(N^{k/2})$. In particular, $\cH_0 = \cH$. 
Furthermore, the scalar product $\la \cdot,\cdot \ra$ on 
$\cD(N^{k/2}) \times \cH$ extends to a sesquilinear map 
$\cH_k \times\cH_{-k} \to \C$ which exhibits $\cH_{-k}$ as the dual space of $\cH_k$. 
We also note that, for every $k \in \Z$, the operator
$N^{1/2} \colon \cD(N^{k/2}) \to \cD(N^{(k-1)/2})$ 
extends to a unitary operator $\cH_k \to \cH_{k-1}$. 
For $n \leq m$, the inclusion $\cD(N^{m/2}) \into \cD(N^{n/2})$, extends to a 
continuous inclusion $\cH_m \into  \cH_n$. 

Any operator $A\in B(\cH_k,\cH_\ell)$ is uniquely determined by its 
restriction to $\cD(N^{k/2})$, so that $B(\cH_k,\cH_\ell)$ can be identified with the 
space of those linear operators $\cD(N^{k/2}) \to \cH_\ell$ extending continuously to 
all of $\cH_k$. Then $\|A\|_{k,\ell}$ denotes the corresponding operator norm. 
This means that, for $A \in B(\cH_{2},\cH)$, we have 
\[ \|Av\| \leq \|A\|_{2,0} \|Nv\| \quad \mbox{ for } \quad v \in \cD(N).\] 

For $A\in B(\cH_k,\cH_{-k})$, $k \in \Z$, we define
\[ [N,A]\in B(\cH_{k+2},\cH_{-k-2}), \quad  [N,A]v:= N(Av)-A(Nv) 
\quad \mbox{ for } \quad v \in \cH_{k+2}.\]
If, for each $v\in \cH_{k+2}$, we have $[N,A]v\in \cH_{-k}$ and there 
exists a $c\geq 0$ with $\|[N,A]v\|_{-k}\leq c\|v\|_k$ for $v\in \cH_{k+2}$, 
we obtain 
by continuous extension an operator, also denoted $[N,A]$, in $B(\cH_k,\cH_{-k})$. 

We will need the following version of Nelson's Commutator Theorem.

\begin{lemma}\label{nelcommthmver0}
Let $A\in B(\cH_1,\cH_{-1})$ such that $[N,A]\in  B(\cH_1,\cH_{-1})$. Then $A\in B(\cH_2,\cH)$ and
\begin{align}\label{nctv0ineq1}
\|A\|_{2,0} \ \leq \ \|A\|_{1,-1}^\frac{1}{2}\cdot (\|A\|_{1,-1}+\|[N,A]\|_{1,-1})^{\frac{1}{2}}.
\end{align}
\end{lemma}

\begin{proof}
By Lemma 1 in Section X.5 of \cite{RS75}, we have $A\in B(\cH_3,\cH_1)$ with 
$$\|A\|_{3,1}\leq \|A\|_{1,-1}+\|[N,A]\|_{1,-1}.$$
Now we apply Proposition~9 in App.~IX.4 of \cite{RS75} 
with $T=N^{-1/2}AN^{-1/2}, B=N^{1/2}$ and $\tilde A = N^{1/2}$ where $\tilde A$ denotes the $A$ in \cite{RS75}. 
This yields $A(\cH_2)\subset \cH$ and \eqref{nctv0ineq1}. Here we use that, for $k \in\Z $,   
the operator $N^{-1/2}:\cH_k\rightarrow \cH_{k+1}$ is unitary. 
\end{proof}

\subsection{Specifying smooth vectors by single elements} 

\begin{definition} \label{def:sembo} Let $G$ be a locally convex Lie group with 
exponential function. We call a smooth unitary representation $(\pi, \cH)$ of $G$ 
{\it semibounded} if the function 
\[ s_\pi \colon \g \to \R \cup \{ \infty\}, \quad 
s_\pi(x) 
:= \sup\big(\Spec(i\oline{\dd\pi}(x))\big) \]
is bounded on a neighborhood of some point $x_0 \in \g$.  
Then the set $W_\pi$ of all such points $x_0$ 
is an open $\Ad(G)$-invariant convex cone. 
Let $I_\pi \subeq \g'$ (the topological dual) be the {\it momentum set of $\pi$}, i.e., 
the weak-$*$-closed convex hull of the linear functionals of the form 
\[ x \mapsto -i \frac{\la \dd\pi(x)v,v\ra}{\la v,v\ra}, \quad 0 \not= v \in \cH^\infty.\] 
Then $s_\pi(x) = \sup \la I_\pi, - x\ra$ 
shows that $s_\pi$ is a lower semi-continuous homogeneous convex function. 
\end{definition}

Before we turn to the main point of this section, we note the following characterization 
of those representations where all operators are smooth. 
\begin{proposition} The identity $\1 \in B(\cH)$ is a smoothing operator if and only if 
$\cH = \cH^\infty$. If the Lie algebra $\g$ is a barreled space, then this 
is equivalent to $\pi \colon G\to \U(\cH)$ being a morphism of Lie groups when 
$\U(\cH)$ carries the norm topology. 
\end{proposition}

\begin{proof} The first assertion is obvious. 
Suppose that $\g$ is a barreled space and $\cH^\infty = \cH$. 
Then $\pi$ is a smooth representation and 
the closed operators $\oline{\dd\pi}(x)$ are everywhere defined, hence bounded 
by the Closed Graph Theorem. As a consequence, the momentum set 
$I_\pi \subeq \g'$ (cf.\ Definition~\ref{def:sembo}) 
is weak-$*$-bounded. If $\g$ is barreled, 
then this implies that $I_\pi$ is equicontinuous, so that 
Theorem~3.1 in \cite{Ne09} shows that $\pi$ is a morphism of Lie groups. 
The converse is clear. 
\end{proof}


\begin{theorem}[Zellner's Smooth Vector Theorem] 
\label{thmsbsmoothonegen} 
Let $\pi:G\rightarrow\U(\cH)$ be a semibounded unitary 
representation and $x_0\in W_\pi$. Then we have $\cH^\infty=\cD^\infty(\overline{\dd\pi}(x_0))$. Moreover, for every continuous seminorm $p$ on $\g$ such that 
\begin{equation}
  \label{eq:ps}
c:=\sup\{s_\pi(x_0+x) : x\in\g, p(x)\leq1\} <\infty 
\end{equation}
and $N:=\frac{1}{i}\overline{\dd\pi}(x_0)+(c+1)\cdot\1$ we have
\begin{align}\label{conjthmestimate}
\|N^k\dd\pi(x)v\|\leq p_k(x)\|N^{k+1}v\| \quad \mbox{ for } \quad 
x\in\g,v\in\cH^\infty,k\in \N_0.
\end{align}
Here $(p_k)_{k \in \N_0}$ is a sequence of continuous seminorms on $\g$ 
defined recursively by 
\[ p_0(x):=p(x)+\frac{1}{2}p([x,x_0]) \quad \mbox{ and } \quad 
p_{k+1}(x):=p_k(x)+p_k([x_0,x]).\] 
\end{theorem}

\begin{proof}
Since $x_0\in W_\pi$, there exists a continuous seminorm $p$ on $\g$ satisfying 
\eqref{eq:ps}. Then $N\geq \1$ and
\begin{equation}\label{svthmeq1}
|\langle \textstyle{\frac{1}{i}}\dd\pi(x)v,v\rangle| \leq c\|v\|^2 + \langle \textstyle{\frac{1}{i}}\dd\pi(x_0)v,v\rangle \leq \langle Nv,v\rangle 
\quad \mbox{ for } \quad 
v\in\cH^\infty, x\in \g, p(x)\leq 1. 
\end{equation}
Hence
\begin{equation}\label{svthmeq2}
|\langle \textstyle{\frac{1}{i}}\dd\pi(x)v,v\rangle| \leq p(x)\langle Nv,v\rangle 
\quad \mbox{ for } \quad 
v\in\cH^\infty, x\in \g. 
\end{equation}
Note that $\dd\pi(x)=0$ for all $x\in\g$ with $p(x)=0$ by \eqref{svthmeq1}. 
In view of Lemma~4.1 in \cite{NZ13}, 
$\cH^\infty \subset \cD^\infty(N)$ 
is dense w.r.t.\ the $C^\infty$-topology on $\cD^\infty(N)$ because it is invariant under 
the unitary one-parameter group $e^{i\R N}$. Thus 
$\overline{N^k\vert_{\cH^\infty}} \supset N^k\vert_{\cD^\infty(N)}$ for $k\in\N$. 
As $\cD^\infty(N)=\cD^\infty(N^{k/2})$ is a core for $N^{k/2}, k\in\N$, and $N\geq\1$, 
we conclude that $\cH^\infty$ is a core for $N^{k/2}$ for all $k\in\N$. In 
particular, $\cH^\infty$ is a core for $N^{1/2}$ and $N^{3/2}$. 
As in Subsection~\ref{subsec:3.1}, let $\cH_n$ denote the Hilbert completion of $\cD(N^{n/2})$ w.r.t. $\|v\|_n:=\|N^{n/2}v\|, n\in\Z$. 

Set $A_x:=\frac{1}{i}\dd\pi(x), x\in\g$. Note that $\cH_1$ is a Hilbert space with 
inner product $\langle Nx,y\rangle$ for $x,y\in\cH^\infty$. Since $\cH^\infty$ is a core for $N^{1/2}$, $\cH^\infty$ is dense in $\cH_1$ and $N^{-1}A_x:\cH_1\rightarrow \cH_1$ is a symmetric operator on $\cH_1$. From \eqref{svthmeq2} we  obtain
\begin{footnote}{Here we use that, for a (densely defined) symmetric 
operator $A \colon \cD(A) \to \cH$ on a Hilbert space $\cH$, we have 
\[ \|A\| = \sup \{ |\la Av,v\ra| \colon v \in \cD(A),\|v\| \leq 1\}.\] 
This relation is well-known if $\cD(A) = \cH$, which implies that $A$ is bounded. 
If $A \geq 0$, then it follows from the Cauchy--Schwarz inequality. 

In general, the relation ``$\geq$'' holds trivially. 
For the converse, assume the right hand side has a finite value $c$. 
Then $A + c \1 \geq 0$ is positive, hence bounded, and this implies that 
$A$ is bounded, so that its closure $\oline A$ is defined on $\cH$. 
As $\cD(A)$ is dense in $\cH$, the assertion now follows from case where 
$A$ is bounded.}  
\end{footnote}
\begin{eqnarray}
  \label{eq:svthmeq3}
\sup_{v\in\cH^\infty, \|v\|_1\leq 1}\|A_xv\|_{-1} 
&=& \sup_{v\in\cH^\infty,\|v\|_1\leq 1}\|N^{-1}A_xv\|_1 \notag \\
&=& \sup_{v\in\cH^\infty,\|v\|_1\leq 1} |\langle NN^{-1}A_xv,v\rangle| \leq p(x).
\end{eqnarray}
Therefore $A_x$ extends to a  bounded 
operator  $\hat A_x\in B(\cH_1,\cH_{-1})$ with 
$\|\hat A_x\|_{1,-1}\leq p(x)$. A priori $[N,\hat A_x]\in B(\cH_3,\cH_{-3})$. Since $[N,A_x]=\dd\pi([x,x_0])=iA_{[x,x_0]}$, 
 we obtain from \eqref{eq:svthmeq3} 
\[ \|[N,\hat A_x]v\|_{-1}\leq p([x,x_0])\|v\|_1\quad \mbox{ for all } \quad 
v\in\cH^\infty.\] 
Since $\cH^\infty\subset \cH_3$ is dense (as $\cH^\infty$ is a core for $N^{3/2}$) we 
conclude that $[N,\hat A_x](\cH_3) \subset \cH_{-1}$ and further
$[N,\hat A_x]\in B(\cH_1,\cH_{-1})$ with $\|[N,\hat A_x]\|_{1,-1} \leq p([x,x_0])$. With Lemma \ref{nelcommthmver0} we obtain that $\hat A_x\in B(\cH_2,\cH)$ and 
$$\|\hat A_x\|_{2,0}\leq p(x)^{\frac{1}{2}}(p(x)+p([x,x_0]))^{\frac{1}{2}}\leq p(x)+\textstyle{\frac{1}{2}}p([x,x_0])$$
Thus 
\begin{eqnarray}
  \label{eq:svthmineq3}
\|\textstyle{\frac{1}{i}}\dd\pi(x)v\| 
&\leq& (p(x)+\textstyle{\frac{1}{2}}p([x,x_0]))\cdot\|Nv\| \notag\\
&\leq& (p(x)+\textstyle{\frac{1}{2}}p([x,x_0]))(\|\dd\pi(x_0)v\|+(c+1)\|v\|)
\end{eqnarray}
for all $x\in \g, v\in \cH^\infty$. Since $\cH^\infty\subset 
\cD^\infty(\overline{\dd\pi}(x_0))$ is dense in the $C^\infty$-topology 
 (Lemma~4.1 in \cite{NZ13}), this estimate 
implies that the map $\dd\pi:\g\times \cH^\infty \rightarrow \cH$ extends uniquely to a continuous bilinear map 
\[ \beta:\g\times \cD^\infty(\overline{\dd\pi}(x_0))\rightarrow \cH, \]
where $\cD^\infty(\overline{\dd\pi}(x_0))$ is equipped with the $C^\infty$-topology w.r.t. $U_t:=\pi(\exp(tx_0))$. Moreover $\beta(x,v)=\overline{\dd\pi}(x)v$ (since $\overline{\dd\pi}(x)$ is closed) and $U_t\beta(x,v)=\beta(\Ad(\exp(tx_0))x,U_tv)$ for all $x\in\g, v\in\cD^\infty(\overline{\dd\pi}(x_0))$. 
Smoothness of the map $t\mapsto U_t\beta(x,v)$ implies that $\beta$ takes values in $\cD^\infty(\overline{\dd\pi}(x_0))$. 
Therefore $\cD^\infty(\overline{\dd\pi}(x_0))\subset \cD^\infty(\pi)$. 
Let $(p_n)_{n \in\N_0}$ be the continuous seminorm from the statement of the theorem. 
 By \eqref{eq:svthmineq3} we have $\|\overline{\dd\pi}(x)v\|\leq p_0(x)\|Nv\|$ for $x\in\g,v\in \cD^\infty(N)=\cD^\infty(\overline{\dd\pi}(x_0))$. We now show inductively that 
\[ \|N^k\overline{\dd\pi}(x)v\|\leq p_k(x)\|N^{k+1}v\| \quad \mbox{ for } \quad 
x\in\g, v\in\cD^\infty(N).\] 
Assume that this holds for some $k\in\N_0$. Then
\begin{eqnarray*}
\|N^{k+1}\overline{\dd\pi}(x)v\|
&\leq &\|N^k\overline{\dd\pi}(x)Nv\|+\|[N,N^k\overline{\dd\pi}(x)]v\| \\
&=& \|N^k\overline{\dd\pi}(x)Nv\|+\|N^k\overline{\dd\pi}([x_0,x])v\|\\
&\leq& p_k(x)\|N^{k+2}v\|+p_k([x_0,x])\|N^{k+1}v\| \leq p_{k+1}(x)\|N^{k+2}v\|
\end{eqnarray*}
for $x\in\g, v\in\cD^\infty(N)$. Since on $\cD^\infty(N)=\cD^\infty(\overline{\dd\pi}(x_0))$ the $C^\infty$-topology w.r.t.\ $N$ coincides with the $C^\infty$-topology w.r.t. $\overline{\dd\pi}(x_0)$, we conclude that the map
$$\beta:\g\times \cD^\infty(\overline{\dd\pi}(x_0))\rightarrow \cD^\infty(\overline{\dd\pi}(x_0)),\quad \beta(x,v)=\overline{\dd\pi}(x)v$$
is continuous. In particular $\g^k\rightarrow\cH,(x_1,\dots,x_k)\mapsto \overline{\dd\pi}(x_1)\cdots\overline{\dd\pi}(x_k)v$ is continuous and $k$-linear for all $k\in\N_0,v\in\cD^\infty(\overline{\dd\pi}(x_0))$. Now Theorem~\ref{thm:2.1}(i) yields $\cD^\infty(\overline{\dd\pi}(x_0))= \cH^\infty$.
\end{proof}

\begin{remark}\label{thmsbsmoothonegenrmk}
Let $\pi:G\rightarrow\U(\cH)$ be a semibounded representation and $x_0\in W_\pi$. 

(a) For $c> s_\pi(x_0)$ the set $M_c:=\{x\in\g : s_\pi(x_0\pm x) < c\}$ is convex, balanced and absorbing. We may choose the seminorm $p$ in the preceding theorem to be the Minkowski functional of $M_c$, i.e.,
\begin{align}\label{minkfunc}
p(x) = \inf\{t>0: s_\pi(x_0\pm \textstyle{\frac{x}{t}})< c\}.
\end{align}
We may also consider $M_c=\{x\in\g : s_\pi(x_0\pm x) < c \ \wedge \ x_0\pm x\in W_\pi\}$, which is open, convex and balanced, and take its Minkowski functional to be the seminorm $p$. This may be helpful to study continuity properties of $p$ if $s_\pi\vert_{W_\pi}$ is known.

(b) Let $\tau$ be the (not necessarily Hausdorff) locally convex topology on $\g$ generated by 
the seminorm $p$ in \eqref{minkfunc}. The proof of the preceding theorem shows that, for all $v\in\cH^\infty$ and $k\in \N$, the map
$$\g^k\rightarrow\cH,\quad (x_1,\dots,x_k)\mapsto \dd\pi(x_1)\cdots\dd\pi(x_k)v$$
is continuous when $\g$ is equipped with the locally convex topology generated by the maps
$$f_k: \g\rightarrow (\g,\tau), \quad x\mapsto \ad(x_0)^kx$$
for $k\in \N_0$.
\end{remark}

\begin{proposition} \label{prop:3.6}
Let $\pi:G\rightarrow\U(\cH)$ be a semibounded representation 
and $x_0\in W_\pi$. We endow $\cH^\infty=\cD^\infty(\overline{\dd\pi}(x_0))$ 
with the $C^\infty$-topology w.r.t.~$\overline{\dd\pi}(x_0)$.
Then the action $G\times\cH^\infty \rightarrow \cH^\infty,(g,v)\mapsto\pi(g)v$ is smooth.  
\end{proposition}

\begin{proof}
Set $B:=\overline{\dd\pi}(x_0)$. The $C^\infty$-topology on $\cD^\infty(B)$ is generated by the seminorms $x\mapsto \|B^kx\|, k\in\N_0$. By the proof of Theorem \ref{thmsbsmoothonegen}, the map 
$$\beta:\g\times \cD^\infty(B)\rightarrow \cD^\infty(B), \quad \beta(x,v)=\dd\pi(x)v$$
is bilinear and continuous. For $k\in\N_0$, the map
$$f_k: G\times \cD^\infty(B)\rightarrow \cH, \quad f_k(g,v)=B^k\pi(g)v=\pi(g)\dd\pi(\Ad(g^{-1})x_0)^kv,$$
is continuous since $\beta$ and the map 
$G\times\cD^\infty(B)\rightarrow\cH,(g,v)\mapsto \pi(g)v$ are continuous. As
$f_k(g,v)=\frac{d^k}{dt^k}\Big\vert_{t=0} \pi(\exp(tx_0)g)v$ and $\cD^\infty(B)=\cH^\infty$ the partial derivatives of $f_k$ exist and are given by
\[\dd f_k(g,v)(g.x,w)=B^k\pi(g)\dd\pi(x)v+B^k\pi(g)w =f_k(g,\beta(x,v))+f_k(g,w)\]
In particular they are continuous. Thus 
$f:G\times\cD^\infty(B)\rightarrow\cD^\infty(B), f(g,v)=\pi(g)v$ is $C^1$ 
and 
\[ \dd f(g,v)(g.x,w)=f(g,\beta(x,v))+f(g,w).\]
We conclude that, if $f$ is $C^k$, then $\dd f$ is also $C^k$, i.e., $f$ is $C^{k+1}$. 
By induction we see that $f$ is smooth.
\end{proof}

\section{Host algebras from smoothing operators} 
\label{sec:4}

In this section we explain how smoothing operators obtained naturally 
from a unitary representation $(\pi, \cH)$ of a metrizable Lie group 
can be used to construct host algebras. Our results are rather 
complete for the class of semibounded representations, 
for which we can build on Theorem~\ref{thmsbsmoothonegen}. 

\begin{definition} Let $G$ be a topological group. 
A {\it host algebra for $G$} is a pair 
$({\cal A}, \eta)$, where  ${\cal A}$ is a $C^*$-algebra and 
$\eta \colon G \to \U(M({\cal A}))$ is a group homomorphism 
such that: 
\begin{itemize}
\item[\rm(H1)] For each non-degenerate representation $(\pi, {\cal H})$ 
of $\cA$, the representation $\tilde\pi \circ \eta$ of $G$ is continuous. Here 
$\tilde\pi \colon M(\cA) \to B(\cH)$ denotes the canonical extension of the non-degenerate 
representation $\pi$ to the multiplier algebra.  
\item[\rm(H2)] For each complex Hilbert space 
${\cal H}$, the corresponding map 
$$ \eta^* \colon \Rep({\cal A},{\cal H}) \to \Rep(G, {\cal H}), \quad 
\pi \mapsto \tilde\pi \circ \eta $$ 
is injective. 
\end{itemize}
\end{definition}

\begin{remark} We recall that, if $\cA \subeq B(\cH)$ is a closed $*$-subalgebra, then 
its multiplier algebra is given concretely by 
\[ M(\cA) \cong \{ B \in B(\cH) \colon B \cA + \cA B \subeq \cA\}\] 
(Section~3.12 in \cite{Pe89}). 
\end{remark}

\begin{lemma} \label{lem:host1} Let $(\pi, \cH)$ be a smooth unitary representation 
of the metrizable Lie group $G$ 
and ${\cB \subeq B(\cH)}$ be a $*$-invariant subset of smoothing operators, i.e., 
$\cB \cH\subeq \cH^\infty$. Let $\cA := C^*(\pi(G)\cB \pi(G))$. Then 
\[ \pi(G) \subeq \{ C \in B(\cH) \colon C \cA + \cA C \subeq \cA \} \cong M(\cA) \] 
which leads to a homomorphism $\eta_G \colon G \to M(\cA)$. For every non-degenerate 
representation $(\rho,\cK)$ of $\cA$, the corresponding representation 
$\tilde\rho \circ \eta_G$ of $G$ is smooth. 
\end{lemma}

\begin{proof} Let $\cC := \pi(G)\cB \pi(G)$, so that 
$\cC^n = \pi(G) (\cB \pi(G))^n =  (\pi(G) \cB)^n \pi(G)$. 
Since all the sets $\cC^n$ are $*$-invariant, $\cA$ is the closed span  of 
$\bigcup_{n = 1}^\infty \cC^n$. Further, $\pi(G) \cC^n + \cC^n\pi(G) \subeq \cC^n$  
implies that $\pi(G) \subeq M(\cA)$, which leads to the homomorphism 
$\eta_G \colon G \to M(\cA)$. 

Next we observe that all operators in $C \in \cC^n, n > 0$ are smoothing, 
so that the maps 
\[ G \to B(\cH), \quad g \mapsto \pi(g) C, \quad C\in \cC^n \] 
are smooth by Theorem~\ref{thm:smoothop}. 
Since $\bigcup_n \cC^n$ spans a dense subspace of $\cA$, we thus obtain a dense subspace 
of smooth vectors for the left multiplier action of $G$ on $\cA$. 
In particular, this action is strongly continuous.

From this property one already derives that every non-degenerate representation 
$(\rho, \cK)$ of $\cA$ defines a smooth representation $\rho_G 
:= \tilde\rho \circ \eta_G \colon G \to \U(\cK)$. 
\end{proof}

\begin{remark} \label{rem:redux} 
The preceding lemma implies that $\cA$ satisfies the condition (H1) of a host algebra. 
To ensure that (H2) is also satisfied, we have to pick the subset $\cB$ in such a way that, 
for two representations $\rho_1, \rho_2 \colon \cA \to B(\cK)$ for which 
$\tilde\rho_1\res_{\eta_G(G)} = \tilde\rho_2\res_{\eta_G(G)}$, we have 
$\rho_1\res_\cB  = \rho_2\res_\cB$. 
\end{remark}

\begin{theorem}[Subgroup Host Algebra Theorem] \label{thm:5.11} 
Let $(\pi, \cH)$ be a unitary representation of the metrizable 
  Lie group $G$ and  $\iota_H \colon H \to G$ a morphism 
of Lie groups where $\dim H < \infty$ and $\pi_H := \pi \circ \iota_H$ 
satisfies 
\[ \cH^\infty = \cH^\infty(\pi_H). \] 
Then $\cA := C^*\big(\pi(G) \pi_H(C^\infty_c(H))\pi(G)\big)$ 
is a host algebra for a class of smooth representations 
$(\rho,\cK)$ of $G$. 
\end{theorem}

\begin{proof} Our assumption implies that the operators 
$\pi_H(f)$, $f \in C^\infty_c(H)$, are smoothing 
(Example~\ref{ex:3.1}(a)). 
We now apply Lemma~\ref{lem:host1} with 
$\cB := \pi_H(C^\infty_c(H))$. 

For a non-degenerate representation $(\rho,\cK)$ of $\cA$, we write 
$\rho_G := \tilde\rho \circ \eta_G$ for the corresponding smooth unitary 
representation of $G$. Since the multiplier action of $H$ on $\cA$ is continuous, 
$\cB\cA$ is dense in $\cA$, so that the representation 
$\rho\res_{\cB}$ is non-degenerate. Next we note that, for $h \in H$ and 
$f \in C^\infty_c(H)$, we have 
\[ \tilde\rho(\pi_H(h))\rho(\pi_H(f)) 
= \rho(\pi_H(h)\pi_H(f))
= \rho(\pi_H(\delta_h * f)),\] 
and this implies that $\rho_G(\iota(h)) = \tilde\rho(\pi_H(h))$ is the 
unique representation of $H$ corresponding to the non-degenerate 
representation $\rho \circ \pi_H$ of the convolution 
$*$-algebra $C^\infty_c(H)$. We conclude that 
\[ \rho(\pi_H(f)) = \int_H f(h) \rho_G(\iota(h)) \, dh \quad \mbox{ for } \quad 
f \in C^\infty_c(H).\] 
In particular, $\rho_G$ determines the representation of $\cA$ and 
thus $\pi \colon G \to  M(\cA)$ defines a host algebra of~$G$. 
\end{proof}

\begin{lemma} \label{lem:5.5} Let $\cA \subeq B(\cH)$ be a $C^*$-subalgebra and 
$T = T^*$ be a selfadjoint operator on $\cH$ bounded from above such that 
\[ e^{T} \in \cA, \quad e^{i\R T} \subeq M(\cA), \] 
and the multiplier action of $\R$ on $\cA$ defined by $U_t := e^{-itT}$ is continuous. 
Then $e^{-izT} \in \cA$ for $\Im z > 0$, and we 
have for every non-degenerate representation $(\rho,\cK)$ of $\cA$ and the 
corresponding unitary one-parameter group $\tilde\rho(U_t)  = e^{-itA}$ the relations 
\[ \rho(e^{-iz T}) = e^{-izA} \quad \mbox{ and } \quad \sup\Spec(A) \leq \sup\Spec(T).\] 
\end{lemma}

\begin{proof} Let $\hat U_z := e^{-izT}$ denote the holomorphic extension of the 
one-parameter group $U$ to the open upper half plane $\C_+$. 
Then $U_{\C_+} \subeq \cA$ follows from 
\[ \hat U_{i + t}  = U_t \hat U_{i} \in \cA \quad \mbox{ for } \quad t \in \R \] 
by analytic continuation. As 
$\cL := C^*(\hat U_{\C_+}) = C^*(U_{L^1(\R)})$  
(Theorem~8.2 in \cite{Ne08}) and the multiplier action of $\R$ given by 
left multiplication with $U_t$ is continuous, we have $\cA = \cL \cA$. 
It follows that any approximate identity of the $C^*$-subalgebra $\cL$ of $\cA$ 
is an approximate identity in $\cA$. 
This implies that, for every representation $(\rho, \cK)$ of $\cA$, the holomorphic 
representation 
\[ \hat V := \rho \circ \hat U \colon \C_+ \to B(\cK) \] 
 is non-degenerate. 
Then $V := \tilde\rho \circ U$ is the corresponding uniquely determined 
multiplier extension to $\R$ because 
\[ V_t \hat V_z =  \rho(U_t \hat U_z) = \rho(\hat U_{t+z})= \hat V_{t + z}.\] 
This implies that the infinitesimal generator 
$A$ of $V$ satisfies 
\[ e^{-izA} = \hat V_z = \rho(\hat U_z) = \rho(e^{-izT})\] 
(cf.\ Proposition~VI.3.2 in \cite{Ne00}). 
We further obtain from 
\[ e^{\sup \Spec A} = \|e^A\|= \|\hat V_{i}\|= \|\rho(\hat U_i)\| 
\leq \|\hat U_i\| = \|e^T\| = e^{\sup \Spec(T)} \] 
the relation $\sup\Spec(A)\leq \sup\Spec(T)$. 
\end{proof}

\begin{theorem} \label{thm:5.1} 
Let $(\pi, \cH)$ be a unitary representation of the 
metrizable Lie group~$G$. If 
\[ \cH^\infty = \cD^\infty(\oline{\dd\pi}(x_0)) \quad \mbox{ for some } \quad x_0 \in \g 
\quad \mbox{ with } \quad 
\sup\Spec(i\oline{\dd\pi}(x_0)) < \infty, \] 
then $\cA := C^*\big(\pi(G) e^{i\oline{\dd\pi}(x_0)} \pi(G)\big)$ 
is a host algebra for a class of smooth representations 
$(\rho_G,\cK)$ of $G$ satisfying 
\begin{equation}
  \label{eq:spec-esti}
\sup\Spec(i\oline{\dd\rho_G}(x_0)) \leq \sup\Spec(i\oline{\dd\pi}(x_0)).
\end{equation}
\end{theorem}

\begin{proof} Our assumption implies that the operator 
$B := e^{i\oline{\dd\pi}(x_0)}$ is bounded. Let 
\[ m := \log\|B\| = \sup\Spec(i\oline{\dd\pi}(x_0))\] 
and $\pi_{x_0}(t) := \pi(\exp tx_0)$. Then the unitary representation 
$\pi_{x_0}$ integrates to a representation of $C^*(\R) \cong C_0(\R)$ which factors 
through the restriction to $C_0(]-\infty,m])$ and 
$B$ is the image of the function $e^t$ under this map. 
Since all functions $t^n e^t$, $n \in \N_0$, vanish at infinity on $]-\infty,m]$, 
the operator $B$ is a smooth vector for the left multiplication action of 
$\R$ on $B(\cH)$ given by $\pi_{x_0}$ (Example~\ref{ex:3.1}(b)). 
Therefore $B$ is smoothing for $\pi_{x_0}$, and since $\cH^\infty  
= \cD^\infty(\oline{\dd\pi}(x_0)) = \cH^\infty(\pi_{x_0})$, 
it is smoothing for $\pi$. We now apply Lemma~\ref{lem:host1} with the one-element set 
$\cB := \{e^{i\oline{\dd\pi}(x_0)}\}$. 

For a non-degenerate representation $(\rho,\cK)$ of $\cA$, we write 
$\rho_G := \tilde\rho \circ \eta_G$ for the corresponding smooth unitary 
representation of $G$. Then Lemma~\ref{lem:5.5} 
implies that 
\[ e^{i\oline{\dd\rho_G}(x_0)} = \rho(e^{i\oline{\dd\pi}(x_0)}).\] 
This formula shows that  $\rho$ is uniquely determined by $\rho_G$, and therefore 
$\cA$ is a host algebra for $G$. 
We further obtain \eqref{eq:spec-esti} from Lemma~\ref{lem:5.5}. 
\end{proof}

\begin{theorem} \label{thm:5.2} 
Let $(\pi, \cH)$ be a semibounded unitary representation of the metrizable 
Lie group $G$, 
then $\cA := C^*\big(\pi(G) e^{i\oline{\dd\pi}(W_\pi)} \pi(G)\big)$ 
is a host algebra for a class of smooth representations 
$(\rho_G,\cK)$ of $G$ satisfying 
\begin{equation}
  \label{eq:spec-esti2}
I_{\rho_G} \subeq I_\pi. 
\end{equation}
\end{theorem}

\begin{proof} First we recall from Theorem~\ref{thmsbsmoothonegen} that 
$\cH^\infty = \cD^\infty(\oline{\dd\pi}(x))$ holds for every $x \in W_\pi$. 
As in the proof of Theorem~\ref{thm:5.1}, we see that 
$\cB  := e^{i\oline{\dd\pi}(W_\pi)}$ consists of smoothing operators. Lemma~\ref{lem:5.5} yields 
\[ e^{i\oline{\dd\rho_G}(x)} = \rho(e^{i\oline{\dd\pi}(x)})\quad \mbox{ for } \quad x \in W_\pi\] 
for every non-degenerate representation $(\rho,\cK)$ of $\cA$. 
This formula shows that  $\rho$ is uniquely determined by $\rho_G$, and therefore 
$\cA$ is a host algebra for $G$. We further obtain from Lemma~\ref{lem:5.5} the relation 
\[ \sup\Spec(i\oline{\dd\rho_G}(x))  \leq \sup\Spec(i\oline{\dd\pi}(x)) 
\quad \mbox{ for } \quad x \in W_\pi.\] 
Since $B(I_\pi)^0 = W_\pi$, Proposition~6.4 in \cite{Ne08} implies that 
\[ I_\pi = \{ \alpha \in \g'\colon (\forall x \in W_\pi)\, 
\alpha(x) \geq \inf \la I_\pi, x \ra \}.\] 
As 
\begin{align*}
 \inf \la I_\pi, x \ra 
&= \inf \Spec(-i\oline{\dd\pi}(x)) 
= -\sup \Spec(i\oline{\dd\pi}(x))\\
&\leq -\sup \Spec(i\oline{\dd\rho_G}(x))
= \inf \la I_{\rho_G}, x\ra,
\end{align*}
we see that, for all $\alpha \in I_{\rho_G}$ and $x \in W_\pi$, we have 
$\alpha(x) \geq  \inf \la I_\pi, x \ra,$ 
hence $\alpha \in I_\pi$, and therefore $I_{\rho_G} \subeq I_\pi$. 
\end{proof}

\begin{corollary} \label{cor:4.9} 
Let $C \subeq \g'$ be a weak-$*$-closed $\Ad^*(G)$-invariant subset
  which is semi-equicontinuous in the sense that its support function 
\[ s_C \colon \g \to \R \cup \{\infty\}, \quad s_C(x) := \sup \la C, x \ra \] 
is bounded in the neighborhood of some point $x_0 \in \g$. Then there exists a 
host algebra $(\cA,\eta)$ of $G$ whose representations correspond to those 
semibounded unitary representations $(\pi, \cH)$ of $G$ for which 
$s_\pi \leq s_C$, i.e., $-I_\pi \subeq C$. 
\end{corollary}

\begin{proof} If $(\pi_j, \cH_j)_{j \in J}$ is a family of semibounded unitary representations 
of $G$ with $I_{\pi_j} \subeq -C$, then their direct $\pi := \oplus_{j \in J} \pi_j$ 
also is a smooth representation which satisfies 
$s_\pi \leq \sup_{j \in J} s_{\pi_j} \leq s_C$, so that we obtain 
$I_\pi \subeq -C$ (Proposition~6.4 in \cite{Ne08}). 

Now let $\fS_C \subeq C^\infty(G,\C)$ be the convex set of all smooth positive 
definite functions $\phi \colon G \to \C$ with $\phi(\1) = 1$ for which the corresponding 
GNS representation $(\pi_\phi, \cH_\phi)$ is semibounded with $s_{\pi_\phi} \leq  s_C$. 
As $\fS_C$ is a set, the direct sum representation $\pi_C := \oplus_{\phi \in \fS_C} \pi_\phi$ 
is defined. It is semibounded with $s_{\pi_C} \leq  s_C$. As every semibounded representation 
$(\rho,\cK)$ with $s_\rho \leq s_C$ is a direct sum of cyclic ones with smooth cyclic 
unit vector, it can be realized in some multiple of the representation~$\pi_C$. 

We now consider the host algebra $\cA_C := C^*\big(\pi_C(G) e^{i\oline{\dd\pi_C}(W_\pi)} \pi_C(G)\big)$ 
from Theorem~\ref{thm:5.2}. Then all representations $\rho$ of $\cA_C$ correspond to 
representations $\rho_G$ of $G$ with $s_{\rho_G} \leq s_{\pi_C}$. Conversely, 
the construction of $\cA_C$ implies that all $G$-representations 
$\pi$ with $s_\pi \leq s_C$ are of the form $\pi = \rho_G$ because this is true for 
all cyclic ones. 
\end{proof}

\subsection{Independence of $\cA$ from the elements $x_0$} 

\begin{lemma}\label{lem1} Let $H$ be a selfadjoint operator on the Hilbert space 
$\cH$ which is bounded from below, $U_t = e^{itH}$ the corresponding 
strongly continuous unitary one-parameter 
group and $A \in B(\cH)$ be a smoothing operator for~$U$. 
Then, for $\C_+=\{z\in\C : \Im z >0\}$, 
the map $f:\overline{\C_+}\to B(\cH), z\mapsto e^{z i H}A$ is continuous.
\end{lemma}

\begin{proof} From Proposition~VI.3.2 in \cite{Ne00} we know 
that $f$ is holomorphic (therefore continuous) on $\C_+$. 
In view of $f(z + t) = U_t f(z)$ for $t \in \R$, 
it suffices to show that $f$ is continuous at $z_0=0$. For $v\in\cH$ we have
\begin{align*}
\|f(z)v-f(0)v\|&=\|e^{ziH}Av-Av\| = \|\int_0^1 e^{tziH}z HAv\, dt\| \\
&\leq |z|\int_0^1\|e^{tziH}\|\cdot\|HAv\|\, dt
\leq |z|e^{(\Im z)\max(0,\sup(\Spec -H))}\|HAv\|.
\end{align*}
Since $A$ is a smoothing operator, $HA$ is bounded 
(Lemma \ref{lem:prod})
and thus 
$f$ is continuous at $0$.
\end{proof}

\begin{proposition}
\label{hadi-prp411}
Let $\pi:G\rightarrow \U(\cH)$ be a semibounded representation of the 
metrizable Lie group~$G$, $x_0\in W_\pi$ and 
$\cA=C^*(\pi(G)e^{i\oline{\dd\pi}(x_0)}\pi(G))$. 
Then  $e^{i\oline{\dd\pi}(x)}\in \cA$ for all $x\in W_\pi$ and, in particular, 
$\cA=C^*(\pi(G)e^{i\oline{\dd\pi}(W_\pi)}\pi(G))$. 
\end{proposition}

\begin{proof} 
We already know that 
$e^{ti\oline{\dd\pi}(x_0)}\in\cA$ for $t>0$ (Lemma~\ref{lem:5.5}). Let $t>0$ and set $y_0:=tx_0$. Let $\lambda \in B(\cH)'$ be a bounded linear functional w.r.t.\ the 
operator norm such that $\lambda\vert_{\cA}=0$. By Lemma~\ref{lem1} the map
\[ f:\overline{\C_+}\to \C,\quad z\mapsto \lambda(e^{z\oline{\dd\pi}(x)}e^{i\oline{\dd\pi}(y_0)})\] 
is continuous and holomorphic on $\C_+$. Moreover $f\vert_\R=0$. Thus $f=0$. 
By the Hahn--Banach Theorem we conclude $e^{i\oline{\dd\pi}(x)}e^{it\oline{\dd\pi}(x_0)}\in\cA$ for all $t>0$. It follows that
\begin{equation}
\label{hadiexpcA}
e^{it\oline{\dd\pi}(x_0)}
e^{i\oline{\dd\pi}(x)}
=\left(e^{i\oline{\dd\pi}(x)}e^{it\oline{\dd\pi}(x_0)}\right)^*\in\cA
\text{ for every }t>0.
\end{equation} 
An argument similar to the proof of Theorem~\ref{thm:5.1}  implies that 
$e^{i\dd\pi(x)}$ is a smoothing operator for the action $\pi_x(t):=\pi(\exp(tx))$ of $\R$ on $\cH$. Consequently, by Theorem~\ref{thmsbsmoothonegen} the operator $e^{i\dd\pi(x)}$ is a smoothing operator for $(\pi,\cH)$. Letting $t\to 0$ in
\eqref{hadiexpcA}, from
Lemma~\ref{lem1} it follows that 
$e^{i\oline{\dd\pi}(x)}\in\cA$.
\end{proof}

\section{Liminality of the constructed $C^*$-algebras} 
\label{sec:5}

If, for an irreducible representation $(\pi, \cH)$ of a $C^*$-algebra $\cA$, 
 the $C^*$-algebra $\pi(\cA)$ contains a non-zero compact operator, then 
$K(\cH) \subeq \pi(\cA)$ (Corollary~4.1.10 in \cite{Dix64}). If
$\pi(\cA) =K(\cH)$
for all irreducible 
representations $(\pi,\cH)$ of $\cA$, then $\cA$ is called {\it liminal}.
\begin{remark} For $\cA = C^*\big(\pi(G)\cB\pi(G)\big)$ as in Lemma~\ref{lem:host1} 
and an irreducible representation 
$(\rho,\cH)$ of $\cA$, we have $\rho(\cA) \subeq K(\cH)$ if and only if 
$\rho(\cB) \subeq K(\cH)$. 
\end{remark}

\begin{example} The host  algebras of the form 
$\cA := C^*\big(\pi(G) \pi_H(C^\infty_c(H))\pi(G)\big)$ 
from Theorem~\ref{thm:5.11} 
are liminal if, for every irreducible representation $(\rho,\cK)$ of $\cA$, the 
representation $\rho_H := \tilde\rho \circ \eta_G \circ \iota_H$ 
maps $C^\infty_c(H)$ into compact operators. This holds in the following situations: 
\begin{itemize}
\item[\rm(a)] If $H$ is a connected reductive Lie group with compact center and 
$\rho_H$ is irreducible (Theorem~6.4 in \cite{GrN14}). 
\item[\rm(b)] If $H$ is abelian and the spectral measure $P$ on $\hat H 
= \Hom(H,\T)$ 
corresponding to $\rho_H$ is a locally finite sum of point measures with finite-dimensional ranges.
For $H = \R$ and $\pi(t) = e^{itA}$, this condition is equivalent to the compactness of the
resolvent $(A + i \1)^{-1}$ (Lemma~C.3 in \cite{GrN14}).
\item[\rm(c)] If $H$ is compact and all irreducible representations of $H$ occur in $\rho_H$ with 
finite multiplicities (Proposition~C.5 in \cite{GrN14}). 
\end{itemize}
\end{example}

\begin{example} The host  algebras of the form 
$\cA := C^*\big(\pi(G) e^{i\oline{\dd\pi}(W_\pi)} \pi(G)\big)$ 
constructed from a semibounded representation in Theorem~\ref{thm:5.2} 
are liminal if for every irreducible semibounded representation $(\rho_G,\cH)$ of $G$, 
the operators $e^{i\oline{\dd\rho_G}(x)}$, $x \in W_\pi$, are compact. This is in particular 
the case if: 
\begin{itemize}
\item[\rm(a)] $G$ is finite dimensional by Theorem~X.4.10 in \cite{Ne00}. 
\item[\rm(b)] $G = \Vir$ is the Virasoro group by Proposition~4.11 in \cite{Ne10a} 
and Proposition~\ref{hadi-prp411} above. See also \cite{NS14} 
for an alternative construction of corresponding $C^*$-algebras. 
\item[\rm(c)] $G$ is a double extension of a twisted loop group 
\[ \cL_\phi(K) = \{ f \in C^\infty(\R, K)\colon f(t+2\pi) = \phi^{-1}(f(t))\},\] 
where $K$ is a semisimple compact Lie group and $\phi$ a finite order automorphism. 
In this case the generator $\bd$ of the translation action is contained in $W_\pi \cup - W_\pi$, 
and this implies that semibounded irreducible representation are positive (or negative) 
energy representations, hence highest weight representations 
(Theorems~5.4 and 6.1 in \cite{Ne14}). 
\end{itemize}
\end{example}

\section{Perspectives} 
\label{sec:6}

We conclude this paper with a discussion of several applications and open 
problems related to smoothing operators and host algebras for infinite 
dimensional Lie groups.

\subsection{Holomorphic induction} 

We have seen in Section~\ref{sec:3} that, for every semibounded 
representation $(\pi, \cH)$ with $x_0 \in W_\pi$, the selfadjoint 
operator $i\oline{\dd\pi}(x_0)$ has the same smooth vectors as $G$. 
Let $P = P[a,b]$ be a spectral projection of this operator corresponding to a 
compact interval in $\R$. Since $P(\cH)$ consists of smooth vectors for 
the one-parameter group $\pi_{x_0}$, the operator $P$ 
is a smoothing operator for $G$. Further $P = P^2 = P^*$ implies that the map 
$G^2 \to B(\cH), (g,h) \mapsto \pi(g)P \pi(h)$ 
is smooth, and in particular the map 
\[ G \to B(\cH), \quad g \mapsto \pi(g)P \pi(g)^{-1} \] 
is smooth. Geometrically, this means that the closed subspace 
$P(\cH)$ has smooth $G$-orbit map in the Gra\ss{}mannian $\Gr(\cH)$ 
of closed subspaces of $\cH$. 

If, in addition, $P(\cH)$ is $G$-cyclic, 
this can be used obtain realizations of the representation 
$(\pi,\cH)$ in spaces of smooth sections of vector bundles over 
$G/G_P$, where $G_P = \{ g \in G \colon
 \pi(g) P  = P \pi(g)\}$ 
(cf.\ \cite{NS14}, \cite{Ne10a}, \cite{Ne14}, \cite{BG14}). 

\subsection{Smoothness and Fr\'echet structure on $\cH^\infty$} 

If $E$ is a locally convex space and $F \subeq E$ a subspace, then 
the Closed Graph Theorem implies that 
$F$ carries at most one Fr\'echet topology for which the inclusion 
$F \into E$ is continuous. 

This implies in particular, that, for a smooth representation 
of a Lie group $G$, the space $\cH^\infty$ carries at most one Fr\'echet 
space topology for which the inclusion 
$\cH^\infty \into \cH$ is continuous. We call $\pi$ {\it Fr\'echet smooth} if it does. 
In Proposition~5.4 of \cite{Ne10b} we have seen that all unitary 
representations of Banach--Lie groups are Fr\'echet smooth. 
As Theorem~\ref{thmsbsmoothonegen} shows, for a semibounded 
representation $(\pi, \cH)$ and $x \in W_\pi$, we have 
$\cH^\infty = \cH^\infty(\pi_x)$. Therefore every semibounded representation 
is Fr\'echet smooth. In both cases the group  $G$ acts smoothly on $\cH^\infty$ 
with respect to the Fr\'echet topology (Theorem~4.4 in \cite{Ne10b} and 
Proposition~\ref{prop:3.6}). 

\nin{\bf Problem:} Is the $G$-action on $\cH^\infty$ smooth for 
every Fr\'echet smooth unitary representation?   

\subsection{Application to dense inclusions of Lie groups} 

The following proposition is sometimes useful to extend semibounded 
representations to larger groups. It demonstrates impressively the power 
of the methods used to prove Theorem~\ref{thmsbsmoothonegen} and that 
the extendibility of a semibounded representation is completely controlled by 
its support function. 

\begin{proposition}
Let $G \to G^\sharp$ be a smooth inclusion of Lie groups for which 
the inclusion $\g\into \g^\sharp$ has dense range and assume that $G$ is connected and $G^\sharp$ is $1$-connected. Let
$(\pi, \cH)$ be a semibounded unitary representation of $G$ for 
which $W_\pi = \g \cap C$ for an open convex cone $C$ in $\g^\sharp$ such that $s_\pi\vert_{W_\pi}$ extends to a continuous map $s_\pi^\sharp:C\to \R$.
Then $\pi$ extends to a semibounded representation $\pi^\sharp$ of $G^\sharp$.
\end{proposition}

\begin{proof} Choose $x_0\in W_\pi$ and $c_0>s_\pi(x_0)$. Then the set 
\[ M:=\{x\in\g \,:\,x_0\pm x\in W_\pi\text{ and } s_\pi(x_0\pm x)<c_0\} \] 
is open, convex, and balanced (cf.~Remark~\ref{thmsbsmoothonegenrmk}), and therefore the Minkowski functional 
\[ p(x):=\inf\left\{t>0\,:\,t^{-1} x \in M \right\} \]
satisfies the assumptions of Theorem~\ref{thmsbsmoothonegen}. Let $c$ be as in 
Theorem~\ref{thmsbsmoothonegen} and 
\[ N=-i\oline{\dd\pi}(x_0)+(c+1)\cdot\1.\] By Theorem~\ref{thmsbsmoothonegen}, $\cH^\infty=\cD^\infty(N)$ and the map
\[ \beta:\g\times \cD^\infty(N) \to \cD^\infty(N),\quad 
(x,v)\mapsto \oline{\dd\pi}(x)v \]  is continuous. Moreover
\begin{align}\label{extpropeq1}
\|N^k\beta(x,v)\|\leq p_k(x)\|N^{k+1}v\|\quad \mbox{ for all } \quad x\in \g,v\in\cH^\infty,k\in\N_0
\end{align}
where $p_k$ are the seminorms defined in Theorem~\ref{thmsbsmoothonegen}.
Next we prove that each seminorm $p_k$ extends uniquely to a continuous seminorm on $\g^\sharp$. From the definition of $p_k$ it is clear that it is enough to show that the seminorm $p$ can be extended to~$\g^\sharp$. As the extension of $p$ we can take the Minkowski functional of 
\[
M^\sharp:=
\left\{x\in\g^\sharp \,:\,x_0\pm x\in C\text{ and } s_\pi^\sharp(x_0\pm x)<c_0\right\}.
\]
In particular, $\beta$ extends (uniquely) to a continuous map 
$\beta^\sharp:\g^\sharp \times \cD^\infty(N)\to \cD^\infty(N)$ also satisfying \eqref{extpropeq1}. Since $\beta$ defines a representation of $\g$ on $\cD^\infty(N)$, by continuity $\beta^\sharp$ defines a representation 
\[ \alpha:\g^\sharp\to \End(\cD^\infty(N)), \quad \alpha(x)v=\beta^\sharp(x,v)\] 
by skew-symmetric operators. Since $[N,\alpha(x)]=i\alpha([x,x_0])$ on $\cD^\infty(N)$, 
 by \eqref{extpropeq1}, we obtain from Theorem~5.2.1 in \cite{TL99} a (unique) continuous representation $\pi^\sharp : G^\sharp \to \U(\cH)$
such that, for every $\gamma\in C^\infty(\R,G^\sharp)$ with $\gamma(0)=\1$ and 
$v\in\cD^\infty(N)$, the curve 
$\pi^{\gamma,v}(t):=\pi^\sharp(\gamma(t))v$ is $C^1$ and satisfies 
$$\frac{d}{dt}\Big\vert_{t=0} \pi^{\gamma,v}(t) = \alpha(\gamma'(0))v.$$
By Lemma~3.3 in \cite{Ne10b} we obtain that each $v\in \cD^\infty(N)$ is a $C^1$-vector for $\pi^\sharp$ and with Theorem \ref{thm:2.1}(i) we further conclude that $\cD^\infty(N)\subset \cH^\infty(\pi^\sharp)$. Thus $\pi^\sharp$ is smooth and since $s_{\pi^\sharp}\vert_C \leq s_\pi^\sharp$ it follows that $\pi^\sharp$ is semibounded. Since $\dd\pi^\sharp\vert_\g=\dd\pi$ and $G$ is connected, $\pi^\sharp\vert_G=\pi$ by 
Proposition~3.4 in \cite{Ne09}.
\end{proof}

\section*{Acknowledgments}

K.-H.~Neeb and C. Zellner acknowledge support by the DFG-grant 
NE 413/7-2 in the framework of the 
Schwerpunktprogramm ``Darstellungstheorie''. 
H. Salmasian acknowledges support by 
the NSERC Discovery Grant RGPIN-2013-355464.

\end{document}